\documentclass[12pt]{amsart}
\usepackage{amsthm,amssymb}
\usepackage{enumerate}
\usepackage[arrow]{xy}
\usepackage{comment}
\usepackage{epsfig}
\usepackage{diagbox}
\usepackage{hyperref}

\usepackage{amssymb,latexsym,cite,epsf,graphics,tikz}
\usetikzlibrary{arrows}

\usepackage{graphicx,color}

\newtheorem{theorem}{Theorem}[section]

\newtheorem{lemma}[theorem]{Lemma}

\newtheorem{proposition}[theorem]{Proposition}
\newtheorem{prop}[theorem]{Proposition}
\newtheorem{corollary}[theorem]{Corollary}

\theoremstyle{remark}

\theoremstyle{definition}

\newtheorem{remark}[theorem]{Remark}
\newtheorem{example}[theorem]{Example}
\newtheorem{definition}[theorem]{Definition}

\def\x{\mathfrak{x}}

\def\R{\mathbb{R}}

\def\Z{\mathbb{Z}}
\def\Q{\mathbb{Q}}

\def\<{{\langle}}
\def\>{{\rangle}}

\def\e{{\epsilon}}

\title{The classification of Zamolodchikov periodic quivers}

\numberwithin{equation}{section}

\begin{document}

\author{Pavel Galashin}
\address{\hspace{-.3in} Department of Mathematics, Massachusetts Institute of Technology,
Cambridge, MA 02139, USA}
\email{galashin@mit.edu}

\author{Pavlo Pylyavskyy}
\address{\hspace{-.3in} Department of Mathematics, University of Minnesota,
Minneapolis, MN 55414, USA}
\email{ppylyavs@umn.edu}

\date{\today}

\thanks{P.~P. was partially supported by NSF grants  DMS-1148634, DMS-1351590, and Sloan Fellowship.}



\begin{abstract}
Zamolodchikov periodicity is a property of certain discrete dynamical systems associated with quivers. 
It has been shown by Keller to hold for quivers obtained as products of two Dynkin diagrams. 
We prove that the quivers exhibiting Zamolodchikov periodicity are in bijection with pairs of 
commuting Cartan matrices of finite type. Such pairs were classified by Stembridge in his study of 
$W$-graphs. The classification includes products of Dynkin diagrams along with four other infinite families, and
eight exceptional cases. We provide a proof of Zamolodchikov periodicity for all four remaining infinite families,
and verify the exceptional cases using a computer program.
\end{abstract}

\maketitle


\section{Introduction}

\subsection{$Y$-systems and $T$-systems associated with recurrent quivers}

A {\it {quiver}} $Q$ is a directed graph without loops and cycles of length two. {\it {Quiver mutations}} are certain transformations one can apply to $Q$ at each of its vertices. One can then associate two levels of algebraic dynamics with quiver mutations: $Y$-variable dynamics and $X$-variable dynamics. 

Of special interest are quivers $Q$ with a bipartite underlying graph having the following property. Color all vertices of $Q$ black or white according to the bipartition. It is well known that mutations of $Q$ (see Section \ref{subsec:quiver_mutations}) at vertices that are not connected by an edge commute. Thus, we can denote 
$\mu_{\bullet}$ (resp., $\mu_{\circ}$) the combined mutations at all black (resp., all white) vertices of $Q$, without worrying about the order in which those are performed. A \emph{recurrent} quiver is a quiver $Q$ with a bipartite underlying graph such that each of $\mu_{\bullet}$ and $\mu_{\circ}$
reverses the directions of arrows in $Q$, but has no other effect on the quiver. Figure \ref{fig:zper1} shows two examples of recurrent quivers. 

  \begin{figure}
    \begin{center}
\vspace{-.1in}
\scalebox{1}{
\input{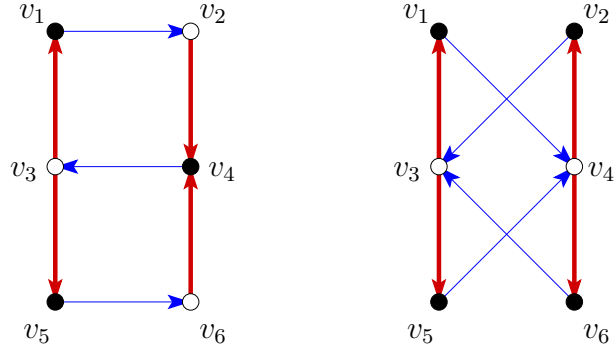} 
}
\vspace{-.1in}
    \end{center} 
    \caption{Two recurrent quivers. The edges are colored according to Section~{\ref{subsec:classification}}}
    \label{fig:zper1}
\end{figure}

\def\Vert{{ \operatorname{Vert}}}
\def\VertQ{{ \Vert(Q)}}
\def\x{{ \mathbf{x}}}

Given a bipartite recurrent quiver $Q$ one can associate to it a {\it {$Y$-system}} and a {\it {$T$-system}}, which are certain infinite systems of algebraic equations defined as follows. Let $\VertQ$ be the set of vertices of $Q$ and define $\x:=\{x_v\}_{v\in\VertQ}$ to be the set of indeterminates, one for each vertex of $Q$. Let $\Q(\x)$ be the field of rational functions in these variables. The \emph{$T$-system} associated with a bipartite recurrent quiver $Q$ is a family $T_v(t)$ of elements of $\Q(\x)$ satisfying the following relations for all $v\in\VertQ$ and all $t\in\Z$ :
\[ T_v(t+1)T_v(t-1)=\prod_{u\to v} T_u(t)+\prod_{v\to w} T_w(t).\]

Here the products are taken over all \emph{arrows} connecting the two vertices. 

Similarly, the \emph{$Y$-system} associated with $Q$ is a family $Y_v(t)$ of elements of $\Q(\x)$ satisfying the following relations for all \emph{white} vertices $v\in\VertQ$ and all $t\in\Z$:
\[ Y_v(t+1)Y_v(t-1)=\prod_{u\to v} (1+Y_u(t))\prod_{v\to w} (1+Y ^{-1}_w(t)) ^{-1}.\]
When $v$ is a \emph{black} vertex, the directions of arrows in the recurrence get reversed:
\[ Y_v(t+1)Y_v(t-1)=\prod_{v\to u} (1+Y_u(t))\prod_{w\to v} (1+Y ^{-1}_w(t)) ^{-1}.\]
We have to give different formulas for vertices of different colors because we prefer not to mutate the edges of the quiver. Alternatively, a single formula would be enough if instead of directions of edges one would use their colors as defined in Section~\ref{subsec:classification}.

We say that a map $ \e:\VertQ\to\{0,1\},\ v\mapsto\e_v$ is a \emph{bipartition} if for every edge $(u,v)$ of $Q$ we have $\e_u\neq \e_v$. We color a vertex $v$ black if and only if $\e_v=1$. Then both $T$- and $Y$-systems associated with $Q$ split into two completely independent parts. Without loss of generality we may consider only one of them. From now on we assume that both systems are defined only for $t\in\Z$ and $v\in\VertQ$ satisfying 
\[t+\e_v\equiv 0\pmod 2.\] 

Both systems are set to the following initial conditions: 
\[T_v(\e_v)=Y_v(\e_v)=x_v\]
for all $v\in\VertQ$. 

\begin{example}
 For the second quiver in Figure \ref{fig:zper1} we have 
 
 \[T_{v_3}(2)= \frac{x_1x_5+x_2x_6}{x_3};\quad T_{v_4}(2)= \frac{x_1x_5+x_2x_6}{x_4};\]
 \[T_{v_2}(3)=\frac{T_{v_3}(2)+T_{v_4}(2)}{x_2}=\frac{x_1x_5}{x_2x_3}+\frac{x_1x_5}{x_2x_4}+\frac{x_6}{x_3}+\frac{x_6}{x_4}.\]
 
\end{example}

Let us say that the $T$-system or the $Y$-system associated with a recurrent quiver $Q$ is {\it {periodic}} if we have 
$T_v(t+2N)=T_v(t)$ (respectively, $Y_v(t+2N)=Y_v(t)$)
for some positive integer $N$ and for all $v\in\VertQ,t\in\Z$ with $t+\e_v$ even. We also refer to a recurrent quiver $Q$ as \emph{Zamolodchikov periodic}
if the associated $T$-system is periodic. 
Equivalently, by Remark~\ref{rem:na}, $Q$ is Zamolodchikov periodic if and only if the associated $Y$-system is periodic.

It is a natural question to ask which bipartite recurrent quivers are Zamolodchikov periodic. In this paper, \emph{we completely classify such quivers}. Almost all literature up until now has been concerned with the tensor product case, explained below. However, the second quiver 
in Figure \ref{fig:zper1} is already Zamolodchikov periodic, but does not arise from the tensor product construction.

\begin{remark}
 We use the term \emph{Zamolodchikov periodic} to indicate the periodicity of the \emph{bipartite} dynamics we consider in this paper. We use this term when referring to quivers. When referring to the specific $T$-systems or $Y$-systems, we just call them periodic. This is in agreement 
 with the more general notion of $T$ and $Y$-systems, introduced by T.~Nakanishi in \cite{Na}.
\end{remark}

\subsection{The tensor product case and its history} \label{sec:tensor}

Let $\Lambda$ and $\Lambda'$ be two $ADE$ Dynkin diagrams. Since $ADE$ Dynkin diagrams are bipartite, we can write:
\[\Vert(\Lambda) = \Lambda_0 \sqcup \Lambda_1;\quad \Vert(\Lambda') = \Lambda'_0 \sqcup \Lambda'_1.\]
Define the {\it {tensor product}} $\Lambda \otimes \Lambda'$, also referred to as {\it {box product}}
in the literature, to be the following quiver.
\begin{itemize}
 \item The vertices of $\Lambda \otimes \Lambda'$ are pairs $(u,w) \in \Vert(\Lambda)\times\Vert(\Lambda')$.
 \item The edges of $\Lambda \otimes \Lambda'$ connect $(u,w)$ and $(u,v)$ if $w$ and $v$ are connected in $\Lambda'$, and also connect $(u,w)$ and $(v,w)$ if $u$ and $v$ are connected in $\Lambda$.
 \item The directions of the the edges in $\Lambda \otimes \Lambda'$ are from $(\Lambda_0, \Lambda'_0)$ to $(\Lambda_1, \Lambda'_0)$, from $(\Lambda_1, \Lambda'_0)$ to $(\Lambda_1, \Lambda'_1)$, from $(\Lambda_1, \Lambda'_1)$ to $(\Lambda_0, \Lambda'_1)$, 
 from $(\Lambda_0, \Lambda'_1)$ to $(\Lambda_0, \Lambda'_0)$.
\end{itemize}
It is easy to check that thus obtained tensor products of simply-laced Dynkin diagrams are bipartite recurrent quivers. 
\begin{example}
 The first quiver in Figure \ref{fig:zper1} is the tensor product $A_3 \otimes A_2$.
\end{example}

The following theorem has a long history, which we will briefly explain. 

\begin{theorem}[B.~Keller, \cite{K}]
 If $\Gamma$ and $\Delta$ are two $ADE$ Dynkin diagrams and $h$ and $h'$ are the associated Coxeter numbers (see Section \ref{sec:ADEDynkin}), then the corresponding $T$- and $Y$-systems are periodic, i.e. they satisfy
 \[T_v(t+2(h+h'))=T_v(t);\quad Y_v(t+2(h+h'))=Y_v(t)\]
  for all $v\in\VertQ,\ t\in\Z$ with $t+\e_v$ even.
\end{theorem}

\begin{remark}
 Keller's theorem is stated for $Y$-systems. The $T$-system version is implicit in \cite{K}, however, since the periodicity of $g$-vectors is proved, see Remark \ref{rem:na}.
\end{remark}

The periodicity of $Y$-systems was first discovered by Zamolodchikov in his study of thermodynamic Bethe ansatz \cite{Z}. His work was in the generality of simply-laced Dynkin diagrams. The conjecture was generalized by Ravanini-Valeriani-Tateo \cite{RVT}, 
Kuniba-Nakanishi \cite{KN}, Kuniba-Nakanishi-Suzuki \cite{KNS}, Fomin-Zelevinsky \cite{FZy}. 
The special cases of the more general tensor product version of the conjecture were proved by Frenkel-Szenes \cite{FS}, Gliozzi-Tateo \cite{GT}, Fomin-Zelevinsky \cite{FZy}, Volkov \cite{Vo}, Szenes \cite{Sz}. In full generality it was proved by Keller \cite{K} and later in a different 
way by  Inoue-Iyama-Keller-Kuniba-Nakanishi \cite{IIKKN1, IIKKN2}.

Besides thermodynamic Bethe ansatz, $T$-systems and $Y$-systems arise naturally in the study of other solvable lattice models, see Kuniba-Nakanishi-Suzuki~\cite{KNS}, Kirillov-Reshetikhin~\cite{KR}, Reshetikhin~\cite{R}, Ogievetsky-Wiegmann~\cite{OW}; in 
the study of Kirillov-Reshetikhin modules and their characters, see Frenkel-Reshetikhin \cite{FR}, Knight \cite{Kni}, Nakajima \cite{N}; as well as in other places. For an excellent survey on appearances of $T$-systems and $Y$-systems in physics, see \cite{KNSi}.

\subsection{Preliminaries}

\subsubsection{$ADE$ Dynkin diagrams and their Coxeter numbers}\label{sec:ADEDynkin} 
By an \emph{$ADE$ Dynkin diagram} we mean a Dynkin diagram of type $A_n,D_n,E_6,E_7,$ or $E_8$. The following characterization of $ADE$ Dynkin diagrams is due to Vinberg \cite{V}:
\begin{theorem}\label{thm:Vinberg}
 Let $G=(V,E)$ be an undirected graph with possibly multiple edges. Then $G$ is an $ADE$ Dynkin diagram if and only if there exists a map $\nu:V\to\R_{>0}$ such that for all $v\in V$, 
 \[2\nu(v)>\sum_{(u,v)\in E} \nu(u).\]
\end{theorem}

For each $ADE$ Dynkin diagram $\Lambda$ there is an associated integer $h(\Lambda)$ called \emph{Coxeter number}. We list Coxeter numbers of $ADE$ Dynkin diagrams in Table~\ref{tab:coxeter}.

\begin{table}
\centering
\begin{tabular}{|c|c|c|c|c|c|}\hline
 $\Lambda$   & $A_n$  & $D_m$  & $E_6$  & $E_7$ & $E_8$\\\hline
$h(\Lambda)$ & $n+1$  & $2m-2$ & $12$   & $18$  & $30$\\\hline
\end{tabular}
\caption{\label{tab:coxeter}Coxeter numbers of $ADE$ Dynkin diagrams}
\end{table}

\def\i{{\mathbf{i}}}

The following construction will be convenient to describe admissible $ADE$ bigraphs below and is going to be an important ingredient of the proof in Section~\ref{sec:ADEtrop}. The only diagrams that have a non-trivial automorphism $\i$ of order $2$ preserving the color of the vertices are the ones of types $A_{2n-1},D_n$, or $E_6$. We are going to take each of these Dynkin diagrams $\Lambda$ and describe a new Dynkin diagram $\Lambda'$. For each vertex $v$ of $\Lambda$ fixed by $\i$, we introduce two vertices $v_+$ and $v_-$ of $\Lambda'$, and for each pair of vertices $(u_+,u_-)$ of $\Lambda$ which are mapped by $\i$ to each other, we introduce just one vertex $u$ of $\Lambda'$. We say that the vertex $v$ in $\Lambda$ \emph{corresponds} to the vertices $v_+,v_-$ in $\Lambda'$ while both of the vertices $u_+,u_-$ in $\Lambda$ \emph{correspond} to the vertex $u$.

  \begin{figure}
    \begin{center}
\vspace{-.1in}
\scalebox{1.1}{
\input{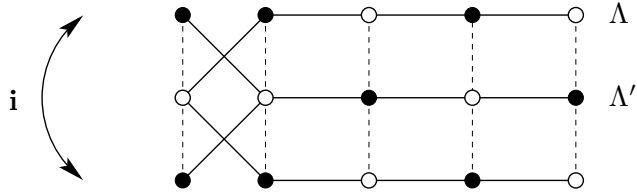} 
}
\vspace{-.1in}
    \end{center} 
    \caption{Dynkin diagrams $A_9$ and $D_6$, the correspondence among the vertices is shown by vertical dashed lines.}
    \label{fig:zper13}
\end{figure}

  \begin{figure}
    \begin{center}
\vspace{-.1in}
\scalebox{1.1}{
\input{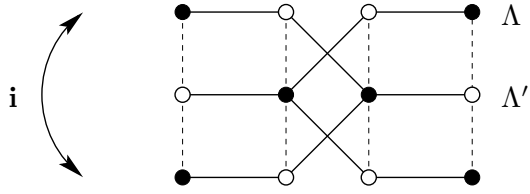} 
}
\vspace{-.1in}
    \end{center} 
    \caption{Two copies of the Dynkin diagram $E_6$, the correspondence among the vertices is shown by vertical dashed lines.}
    \label{fig:zper14}
\end{figure}

For two vertices $u',v'$ of $\Lambda'$ there is an edge between them if and only if there is an edge $(u,v)$ in $\Lambda$ such that $u$ corresponds to $u'$ and $v$ corresponds to $v'$. 

This construction transforms $A_{2n-1}$ into $D_{n+1}$ and vice versa, see Figure~\ref{fig:zper13}; it transforms $E_6$ into itself and induces a nontrivial correspondence between its vertices. We illustrate this correspondence in Figure \ref{fig:zper14}.

\begin{remark}\label{remark:coxeter}
 This is a baby version of the construction in Section~\ref{sec:ADEtrop}. Note that whenever two $ADE$ Dynkin diagrams are related by such a construction, their Coxeter numbers coincide: \[h(A_{2n-1})=h(D_{n+1})=2n.\]
\end{remark}

\subsubsection{Admissible $ADE$ bigraphs}\label{subsec:classification}
In \cite{S} Stembridge studies admissible $W$-graphs for the case when $W=I(p)\times I(q)$ is a direct product of two dihedral groups. These $W$-graphs encode the structure of representations of Iwahori-Hecke algebras, and were first introduced by Kazhdan and Lusztig in \cite{KL}.
The following definitions are adapted from \cite{S} with slight modifications (discussed in Section~\ref{subsec:2by2}). A \emph{bigraph} 
is an ordered pair of simple (undirected) graphs $(\Gamma, \Delta)$ which share a common set of vertices $V:=\Vert(\Gamma)=\Vert(\Delta)$ 
and do not share edges. A bigraph is called \emph{bipartite} if there is a map $\e:V\to\{0,1\}$ such that for every edge $(u,v)$ of $\Gamma$ or of $\Delta$ we have $\e_u\neq \e_v$. 

\def\Bigraph{{G}}
\def\Quiver{{Q}}

There is a simple one-to-one correspondence between bipartite quivers and bipartite bigraphs. Namely, to each bipartite quiver $Q$ with a bipartition $\e:\VertQ\to\{0,1\}$ we  
associate a bigraph $G(Q)=(\Gamma(Q),\Delta(Q))$ on the same set of vertices defined as follows:
\begin{itemize}
 \item $\Gamma(Q)$ contains an (undirected) edge $(u,v)$ if and only if $Q$ contains a directed edge $u\to v$ with $\e_u=0,\e_v=1$;
 \item $\Delta(Q)$ contains an (undirected) edge $(u,v)$ if and only if $Q$ contains a directed edge $u\to v$ with $\e_u=1,\e_v=0$. 
\end{itemize}
Similarly, we can direct the edges of any given bipartite bigraph $G$ to get a bipartite quiver $Q(G)$.

We will be interested in a particular type of bigraphs. An \emph{admissible $ADE$ bigraph} is a bipartite bigraph such that
\begin{itemize}
 \item each connected component of $\Gamma$ is an $ADE$ Dynkin diagram;
 \item each connected component of $\Delta$ is an $ADE$ Dynkin diagram;
 \item the adjacency $|V|\times|V|$ matrices $A_\Gamma$ and $A_\Delta$ of $\Gamma$ and $\Delta$ commute. 
\end{itemize}

One important property of  admissible $ADE$ bigraphs is the following:

\begin{proposition}[{\cite[Corollary 4.4]{S}}]\label{prop:coxeter}
 In a connected admissible ADE bigraph $(\Gamma,\Delta)$, there exist positive integers $p$ and $q$ such that every connected component of $\Gamma$ has Coxeter number $p$ and every connected component of $\Delta$ has Coxeter number $q$.
\end{proposition}
For a bipartite quiver $Q$ such that $G(Q)$ is an admissible $ADE$ bigraph, we denote $h(Q):=p$ and $h'(Q):=q$ to be the corresponding Coxeter numbers given by this proposition.

It is convenient to think of $(\Gamma, \Delta)$ as of a single graph with edges of two colors: red for the edges of $\Gamma$ and blue for the edges of $\Delta$.

The problem of classifying admissible $ADE$ bigraphs is equivalent to that of classifying pairs of (possibly reducible) commuting Cartan matrices of finite type 
satisfying two minor additional conditions equivalent to $\Gamma$ and $\Delta$ not sharing any edges and $G$ being bipartite.

Given a bigraph $(\Gamma,\Delta)$, its \emph{dual} is the bigraph $(\Delta,\Gamma)$ on the same set of vertices. It is obtained from the original bigraph by switching the colors of all edges, and on the level of bipartite quivers this corresponds to reversing the directions of all arrows. 

In \cite{S}, Stembridge classifies admissible $ADE$ bigraphs. He proves that there are, up to taking duals, exactly five infinite families, as well as eight exceptional cases. Here we list the infinite families. For the exceptional cases we refer the reader to Appendix~\ref{sec:app}.

{\bf {Tensor products}}. This is the familiar to us class of bigraphs that are obtained from quivers defined in Section \ref{sec:tensor}. 
The first quiver in Figure \ref{fig:zper1} is the tensor product $A_3 \otimes A_2$. 
A bit more sophisticated example of $D_5 \otimes A_3$ is shown in Figure \ref{fig:zper2}. 

  \begin{figure}
    \begin{center}
\vspace{-.1in}
\scalebox{1}{
\input{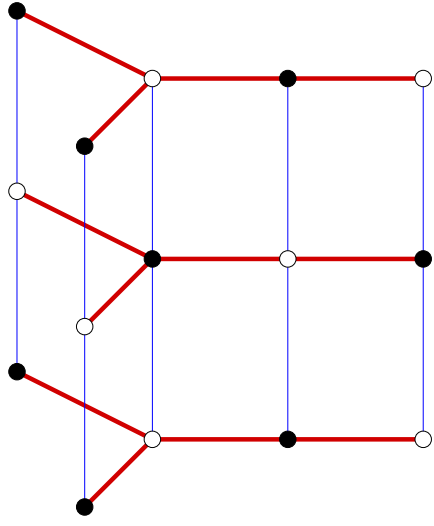} 
}
\vspace{-.1in}
    \end{center} 
    \caption{Tensor product $D_5 \otimes A_3$.}
    \label{fig:zper2}
\end{figure}

{\bf {Twists}}. Twists are denoted $\Lambda \times \Lambda$, where $\Lambda$ is an $ADE$ Dynkin diagram. Twists are obtained by taking two red copies $\Lambda_1$ and $\Lambda_2$ of the Dynkin diagram and connecting $u_1 \in \Lambda_1$ and $w_2 \in \Lambda_2$ by a blue edge if and only if 
$u_1$ and the corresponding vertex $w_1 \in \Lambda_1$ are connected. This way, both $\Gamma$ and $\Delta$ consist of two copies of $\Lambda$. 
The second quiver in Figure \ref{fig:zper1} corresponds to the twist $A_3 \times A_3$. 
An example of $D_5 \times D_5$ is shown in Figure \ref{fig:zper3}. 

  \begin{figure}
    \begin{center}
\vspace{-.1in}
\scalebox{1}{
\input{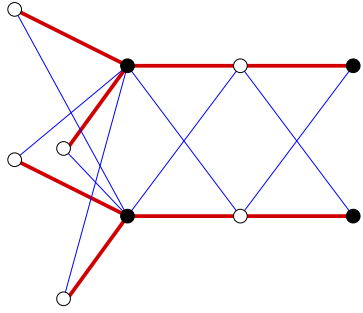} 
}
\vspace{-.1in}
    \end{center} 
    \caption{Twist $D_5 \times D_5$.}
    \label{fig:zper3}
\end{figure}

{\bf {The $(AD^{m-1})_n$ family}}. This family is obtained by taking $A_{m-1} \otimes D_{n+1}$ and connecting the last (blue) copy of $D_{n+1}$ by red edges to an additional blue component $A_{2n-1}$. More precisely, the vertices of $D_{n+1}$ are connected to the \emph{corresponding} vertices of $A_{2n-1}$ in the sense of Section \ref{sec:ADEDynkin}. This family is self dual, when one swaps the red and blue colors in $(AD^{m-1})_n$ one obtains $(AD^{n-1})_m$. An example of $(AD^{3})_5$ which is dual to $(AD^4)_4$ is shown in Figure \ref{fig:zper5}. Note that Proposition~\ref{prop:coxeter} applies here: the red connected components of $(AD^{3})_5$ are $A_7$ and $D_5$ with $h(A_7)=h(D_5)=8$ while the blue connected components of $(AD^{3})_5$ are $A_9$ and $D_6$ with $h(A_9)=h(D_6)=10$. Thus if $G$ is the bigraph of type $(AD^{3})_5$ then $h(G)=8$ and $h'(G)=10$.

  \begin{figure}
    \begin{center}
\vspace{-.1in}
\scalebox{1}{
\input{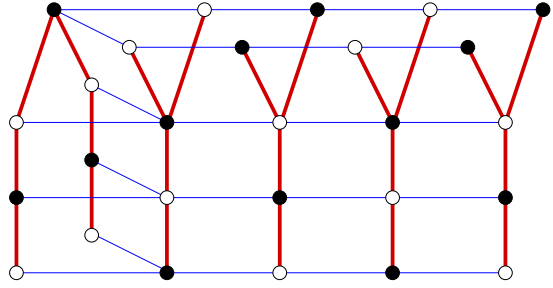} 
}
\vspace{-.1in}
    \end{center} 
    \caption{The case $(AD^{3})_5 \sim (AD^{4})_4$.}
    \label{fig:zper5}
\end{figure}

{\bf {The $(A^{m-1}D)_n$ family}}. This family is obtained by taking $A_{m-1} \otimes A_{2n-1}$ and connecting the last (blue) copy of $A_{2n-1}$ by red edges to an additional blue component $D_{n+1}$, again, via the correspondence described in Section~\ref{sec:ADEDynkin}.
This family is self dual, when one swaps red and blue colors in $(A^{m-1}D)_n$ one obtains $(A^{n-1}D)_m$. An example of $(A^{3}D)_5$ which is dual to $(A^{4}D)_4$ is shown in Figure \ref{fig:zper6}. 

  \begin{figure}
    \begin{center}
\vspace{-.1in}
\scalebox{1}{
\input{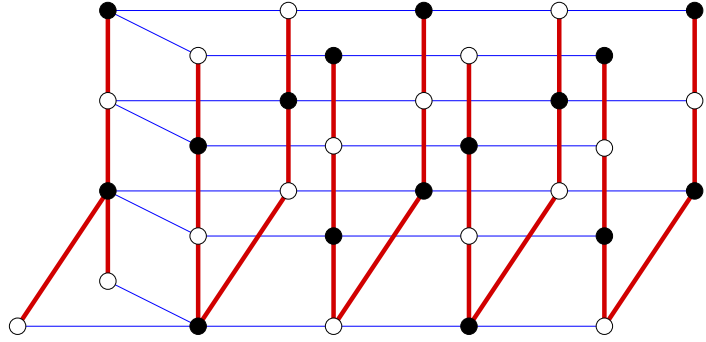} 
}
\vspace{-.1in}
    \end{center} 
    \caption{The case $(A^{3}D)_5 \sim (A^{4}D)_4$.}
    \label{fig:zper6}
\end{figure}

{\bf {The $EE^{n-1}$ family}}. This family is obtained by taking $E_6 \otimes A_{n-1}$ and connecting the last (red) copy of $E_6$ by blue edges to an additional red component $E_6$ via the correspondence described in Section~\ref{sec:ADEDynkin}.
The blue edges then form two components of type $A_{2n-1}$ and another two components of type $D_{n+1}$, and the middle $A_{2n-1}$ is connected to the middle $D_{n+1}$ via the construction of Section~\ref{sec:ADEDynkin}. The dual of this family is denoted by $A_{2n-1}\equiv A_{2n-1}*D_{n+1}\equiv D_{n+1}$ in \cite{S}, however, dual bigraphs yield the same $T$- and $Y$- systems (see Section~\ref{subsec:bigraph_quiver_translation}), so we only need to consider one of these two families. 

An example of $EE^4$ is shown in Figure \ref{fig:zper7}. 

  \begin{figure}
    \begin{center}
\vspace{-.1in}
\scalebox{1}{
\input{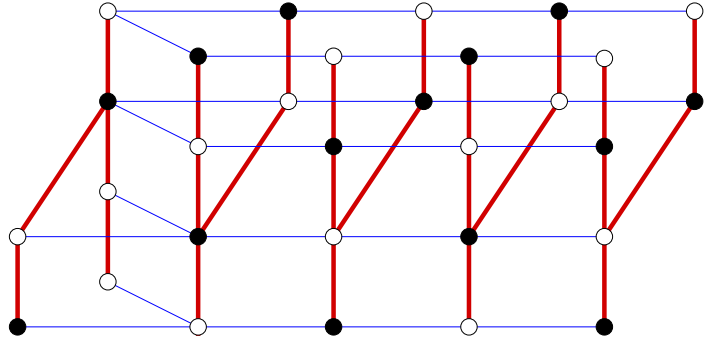} 
}
\vspace{-.1in}
    \end{center} 
    \caption{The case $EE^4$.}
    \label{fig:zper7}
\end{figure}

\subsubsection{The strictly subadditive labeling property}\label{subsec:subadditive}

Let $Q$ be a recurrent quiver. A {\it {labeling}} of its vertices is a function $\nu: \VertQ \rightarrow \mathbb R_{>0}, \;$ which assigns to each vertex $v$ of $Q$ a positive real label $\nu(v)$.
We say that a labeling $\nu$ is {\it {strictly subadditive}} if for any vertex $v\in\VertQ$, the following two inequalities hold:
$$2 \nu(v) > \sum_{u \rightarrow v} \nu(u), \; \text{ and }  \; 2 \nu(v) > \sum_{v \rightarrow w} \nu(w).$$
As before, the sums here are taken over all possible arrows. 
  \begin{figure}
    \begin{center}
\vspace{-.1in}
\scalebox{1}{
\input{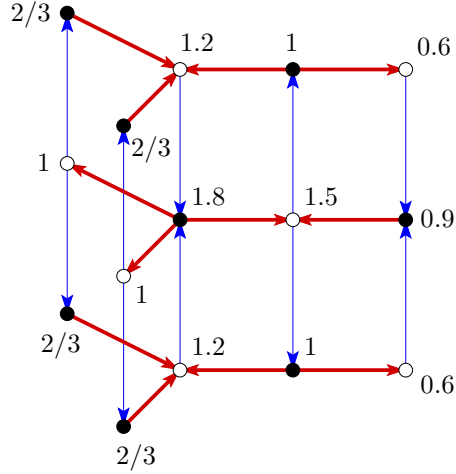} 
}
\vspace{-.1in}
    \end{center} 
    \caption{A strictly subadditive labeling of $D_5 \otimes A_3$.}
    \label{fig:zper8}
\end{figure}

An example of a strictly subadditive labeling of a recurrent quiver $Q(D_5 \otimes A_3)$ is given in Figure \ref{fig:zper8}. We say that a recurrent quiver $Q$ has {\it {the strictly subadditive labeling property}} if there exists a strictly subadditive labeling $\nu:\VertQ\to\R_{>0}$.

Strictly subadditive labelings of quivers have been introduced in \cite{P}. The terminology is motivated by Vinberg's {\it {subadditive labelings}} \cite{V} for non-directed graphs (see Theorem~\ref{thm:Vinberg}).

\subsubsection{The fixed point property}\label{subsec:fixpt}

Consider a $T$-system associated with a recurrent quiver. We say that a particular assignment 
$\rho: \VertQ \to\R_{>1}$
of real values greater than $1$ to the vertices of $Q$ is a \emph{fixed point} if when we substitute $x_u:=\rho(u)$ for all $u\in\VertQ$ into the $T$-system, it becomes time-independent:
\[T_v(t)\mid_{\x:=\rho}=\rho(v)\]
for all $t\in\Z$ and $v\in \VertQ$ with $t+\e_v$ even. 

In other words, $\rho$ is a \emph{fixed point} if for all $v\in\VertQ$ it satisfies
\[\rho(v)^2=\prod_{u\to v}\rho(u)+\prod_{v\to w}\rho(w).\]

\begin{figure}
    \begin{center}
\vspace{-.1in}
\scalebox{1}{
\input{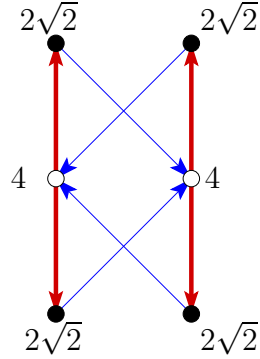} 
}
\vspace{-.1in}
    \end{center} 
    \caption{A fixed point of $Q(A_3 \times A_2)$.}
    \label{fig:zper9}
\end{figure}

An example of a fixed point for a recurrent quiver $Q(A_3 \times A_2)$ is given in Figure \ref{fig:zper9}.  
We say that a recurrent quiver has {\it {the fixed point property}} if there exists a fixed point $\rho:\VertQ\to\R_{>1}$.

\subsubsection{Tropical $T$-systems}

\def\l{{ \lambda}}
\def\Ttr{ \mathfrak{t}^\l}
\def\t{ \mathfrak{t}}

We define yet another system of algebraic relations associated with a bipartite recurrent quiver $Q$. 

Let $\l: \VertQ\to\R$ be any map. The \emph{tropical $T$-system} $\t^\l$ associated with $Q$ is a family $\Ttr_v(t)\in\R$ of real numbers satisfying the following relation:
\begin{equation}\label{eq:tropical}
\Ttr_v(t+1)+\Ttr_v(t-1)=\max\left(\sum_{u\to v}\Ttr_u(t),\sum_{v\to w}\Ttr_w(t)\right).
\end{equation}

As before, we impose some initial conditions:
\[\Ttr_v(\e_v)=\l(v)\]
for all $v\in\VertQ$ and we again consider only the subsystem defined for $t,v$ such that $t+\e_v$ is even.

It is easy to see that this system is just the \emph{tropicalization} of the $T$-system associated with $Q$ (see Section~\ref{sec:pertrop}).

\begin{example}\label{example:A3_A1}
 Let $Q$ be a Dynkin diagram of type $A_3\otimes A_1$ with $\VertQ=(a,b,c)$ and $\e_a=\e_c=0,\ \e_b=1$. Let $\l:\VertQ\to\R$ be defined by
 \[\l(a)=3,\ \l(b)=-2,\ \l(c)=7.\]
 Since both of the edges of the bigraph $G(Q)=(\Gamma,\emptyset)$ are of the same color, the defining recurrence (\ref{eq:tropical}) of the tropical $T$-system for $Q$ becomes just
 
\[\Ttr_v(t+1)+\Ttr_v(t-1)=\max\left(\sum_{(u,v)\in \Gamma}\Ttr_u(t),0\right).\]

 The values $\Ttr_v(t)$ of the tropical $T$-system $\Ttr$ are given in Table~\ref{tab:tropical_example}, defined only for $t+\e_v$ even.
 
\begin{table}
\centering
\begin{tabular}{|c|ccc|}\hline
\backslashbox{$t$}{$v$}     &  $a$  & $b$   & $c$   \\\hline
 $13$      &       & $-2$  &       \\\hline
 $12$      & $ 3$  &       & $ 7$  \\\hline
 $11$      &       & $12$  &       \\\hline
 $10$      & $ 9$  &       & $ 5$  \\\hline
 $9$       &       & $ 2$  &       \\\hline
 $8$       & $-7$  &       & $-3$  \\\hline
 $7$       &       & $-2$  &       \\\hline
 $6$       & $ 7$  &       & $ 3$  \\\hline
 $5$       &       & $12$  &       \\\hline
 $4$       & $ 5$  &       & $ 9$  \\\hline
 $3$       &       & $ 2$  &       \\\hline
 $2$       & $-3$  &       & $-7$  \\\hline
 $1$       &       & $-2$  &       \\\hline
 $0$       & $ 3$  &       & $ 7$  \\\hline
\end{tabular}
\caption{\label{tab:tropical_example}The values of $\Ttr_v(t)$.}
\end{table}

Here we can see that $\Ttr_v(t+12)=\Ttr_v(t)$ for all $v\in\VertQ$. This is a special case of the tropical version of Zamolodchikov periodicity (see Theorem~\ref{thm:main}), where $12=2(h(Q)+h'(Q))$ with $h(Q)=h(A_3)=4$ and $h'(Q)=h(A_1)=2$.
\end{example}

Let us say that a tropical $T$-system associated with a recurrent quiver $Q$ is {\it {periodic}} if we have $\Ttr_v(t+2N)=\Ttr_v(t)$ for some positive integer $N$ and for all $t\in\Z$ and $v\in\VertQ$ with $t+\e_v$ even. 

\subsection{The main theorem and the outline of the proof}

The following theorem is the main result of this paper. 

\begin{theorem} \label{thm:main}
 Let $Q$ be a bipartite recurrent quiver. Then the following are equivalent.
 \begin{enumerate}
  \item  \label{it:admissibleADE} $G(Q)$ is an admissible $ADE$ bigraph.
  \item  \label{it:strictlySubadditive} $Q$ has the strictly subadditive labeling property.
  \item  \label{it:fixedPoint} $Q$ has the fixed point property.
  \item \label{it:tropicalPeriod} The tropical $T$-system associated with $Q$ is periodic for any initial value $\l:\VertQ\to\R$.
  \item \label{it:birationalPeriod} The $T$-system associated with $Q$ is periodic. 
 \end{enumerate}
 In all cases, we have 
   \[T_v(t+2(h+h'))=T_v(t);\quad Y_v(t+2(h+h'))=Y_v(t);\quad \Ttr_v(t+2(h+h'))=\Ttr_v(t)\]
    for all $\l:\VertQ\to\R$ and for all $v\in\VertQ$, $t\in\Z$ with $t+\e_v$ even. Here $h=h(Q)$ (resp. $h'=h'(Q)$) is the common Coxeter number of all red (resp. blue) connected components of $G(Q)$ provided by Proposition~\ref{prop:coxeter}. 
\end{theorem}

Now, we are ready to formulate the plan of the paper. First, we discuss the connection between quivers, bigraphs and Stembirdge's definitions in Section \ref{sec:ADE}. 
Then, we prove Zamolodchikov periodicity for the case when $G(Q)$ is a twist $\Lambda\times \Lambda$ in Section~\ref{sec:twist}. After that, we proceed with the proof of the theorem according to the logic explained in Figure \ref{fig:zper10}. 

  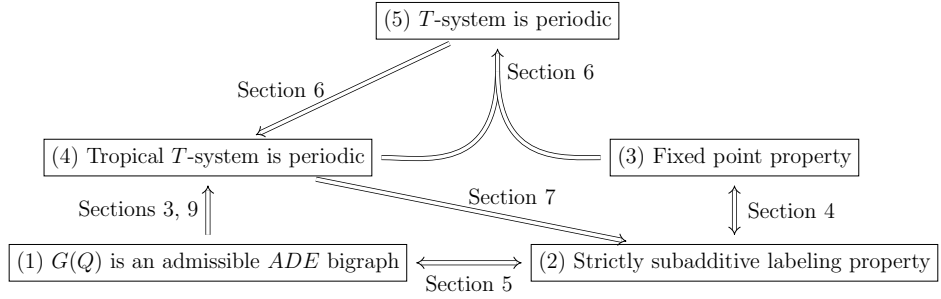
\begin{figure}
    \begin{center}

\scalebox{0.7}{
\begin{tikzpicture}[yscale=2]
 \node[anchor=south,draw,anchor=center] (5) at (0,5.3) {(5) $T$-system is periodic};
 \node[anchor=east,draw,anchor=center] (4) at (-5.5,4) {(4) Tropical $T$-system is periodic};
 \node[anchor=west,draw,anchor=center] (3) at (4.5,4) {(3) Fixed point property};
 \node[anchor=west,draw,anchor=center] (2) at (4.5,3) {(2) Strictly subadditive labeling property};
 \node[anchor=east,draw,anchor=center] (1) at (-5.5,3) {(1) $G(Q)$ is an admissible $ADE$ bigraph};
 
  \coordinate (zJoin) at ([yshift=-0.3cm]5.south);
 
        \draw[>=implies,->,double equal sign distance,shorten >= 5pt] (zJoin) -- (5) node [pos=0, above,anchor=west,inner sep=5pt] (S66) {Section~\ref{sec:pertrop}};
        \path (4) edge [out=0, in=270,>=implies,-,double equal sign distance,shorten <=5pt] (zJoin);
        \path (3) edge [out=180, in=270,>=implies,-,double equal sign distance,shorten <=5pt] (zJoin);
\draw[>=implies,->,double equal sign distance,shorten >= 5pt, shorten <=5pt] (5) -- (4) node [midway, above,anchor=east,inner sep=15pt] (S6) {Section~\ref{sec:pertrop}};
\draw[>=implies,->,double equal sign distance,shorten >= 5pt, shorten <=5pt] (1) -- (4) node [midway, above,anchor=east,inner sep=5pt] (S39) {Sections~\ref{sec:twist},~\ref{sec:ADEtrop}};
\draw[>=implies,<->,double equal sign distance,shorten >= 5pt, shorten <=5pt] (2) -- (3) node [midway, above,anchor=west,inner sep=7pt] (S4) {Section~\ref{sec:fixsub}};
\draw[>=implies,<->,double equal sign distance,shorten >= 5pt, shorten <=5pt] (2) -- (1) node [midway, below,anchor=north,inner sep=7pt] (S5) {Section~\ref{sec:subADE}};
\draw[>=implies,->,double equal sign distance,shorten >= 8pt, shorten <=8pt] (4) -- (2) node [pos=0.3, right,anchor=west,inner sep=35pt] (S7) {Section~\ref{sec:tropsub}};

\end{tikzpicture}}

    \end{center} 
    \caption{The plan of the proof.}
    \label{fig:zper10}
\end{figure}

In Sections \ref{sec:fixsub}, \ref{sec:subADE}, \ref{sec:pertrop}, and \ref{sec:tropsub} we give uniform proofs (that is, independent on Stembridge's classification) of all implications in Figure~\ref{fig:zper10} except for (\ref{it:admissibleADE})$\Longrightarrow$(\ref{it:tropicalPeriod}). In Section~\ref{sec:AA}, as an example, we apply our machinery to reprove the periodicity in the $AA$ case first shown by Volkov in~\cite{V}. Next, we introduce \emph{symmetric labeled bigraphs} in Section~\ref{subsec:involutions} which we then apply to give a case-by-case proof of (\ref{it:admissibleADE})$\Longrightarrow$(\ref{it:tropicalPeriod}) in Section~\ref{subsec:reduction}. Note that we do not reprove periodicity for tensor products of Dynkin diagrams, instead we reduce the remaining three infinite families from Stembridge's classification to the case of tensor products. 
Finally, we discuss the exceptional cases and the corresponding computer check in Appendix~\ref{sec:app}.

\begin{remark} \label{rem:na}
 It is well established in the literature that the periodicities of $T$-systems and that of $Y$-systems are equivalent. The logic consists of three steps.
 \begin{itemize}
  \item The periodicity of the $C$-matrix follows from the periodicity of $Y$-variables, essentially by definition. 
  \item The periodicity of the $G$-matrix is equivalent to the periodicity of $X$-variables, and in fact of the whole seed, as was shown in \cite[Theorem 5.1]{IIKKN1}, relying on the work of Plamondon \cite{Pl}.
  \item The $G$- and $C$-matrices are \emph{duals}, i.e. one of them is the transpose of the inverse of the other, see \cite[Theorem 4.1]{N}.
 \end{itemize}
We thank Tomoki Nakanishi for pointing this out to us, and for other useful comments on the draft of the paper. 
\end{remark}

\section{Quiver mutations, commuting Cartan matrices, and admissible $ADE$ bigraphs} \label{sec:ADE}

\subsection{Quiver mutations}\label{subsec:quiver_mutations}

A {\it {quiver}} $Q$ is a directed graph without loops or directed $2$-cycles. For a vertex $v$ of $Q$ one can define the \emph{quiver mutation $\mu_v$ at $v$} as follows:
\begin{itemize}
 \item for each pair of edges $u \rightarrow v$ and $v \rightarrow w$ create an edge $u \rightarrow w$;
 \item reverse the direction of all edges adjacent to $v$;
 \item if some directed $2$-cycle is present, remove both of its edges; repeat until there are no more directed $2$-cycles.
\end{itemize}

Let us denote the resulting quiver $\mu_v(Q)$. 

  \begin{figure}
    \begin{center}
\vspace{-.1in}
\scalebox{1}{
\input{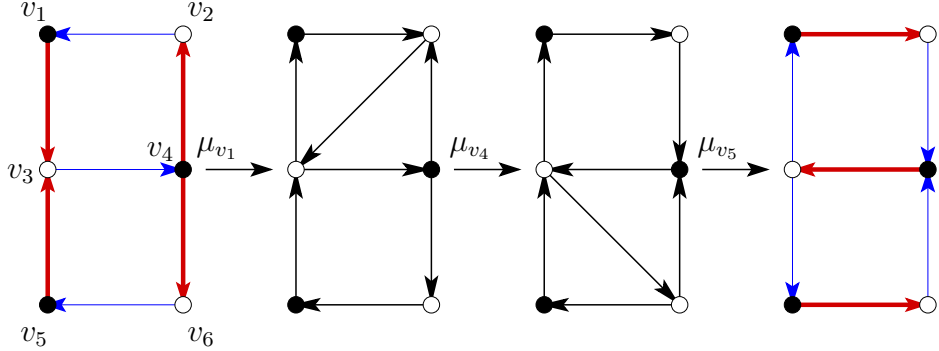} 
}
\vspace{-.1in}
    \end{center} 
    \caption{Three mutations of $Q(A_3 \otimes A_2)$.}
    \label{fig:zper12}
\end{figure}
Mutations in the case $A_3 \otimes A_2$ are illustrated in Figure \ref{fig:zper12}. 

Now, let $Q$ be a bipartite quiver. Recall that $\mu_\circ$ (resp., $\mu_\bullet$) is the simultaneous mutation at all white (resp., all black) vertices of $Q$. The following proposition is transparent from the definition.

\begin{proposition}
 For a bipartite quiver $Q$, if $\e_u\neq \e_v$ then the number of edges $u\to v$ in $Q$ equals the number of edges $v\to u$ in $\mu_\circ(Q)$. Otherwise, if $u,v$ are both black (resp., white), then the number of edges $u\to v$ in $\mu_\circ(Q)$ (resp., in $\mu_\bullet(Q)$) is equal to the number of directed paths of length $2$ from $u$ to $v$ in $Q$.
\end{proposition}

\begin{corollary}
 A bipartite quiver $Q$ is recurrent if and only if for any two vertices $u,v$ of the same color, the number of directed paths of length $2$ from $u$ to $v$ equals the number of directed paths  of length $2$ from $v$ to $u$. 
\end{corollary}

Translating this property into the language of bigraphs yields the following characterization:
\begin{corollary}\label{cor:recurrent_commuting}
 A bipartite quiver $Q$ is recurrent if and only if the associated bipartite bigraph $G(Q)$ has commuting adjacency matrices $A_\Gamma,A_\Delta$. 
\end{corollary}
\begin{proof}
 This follows immediately from the observation that if $u$ and $v$ are both black (resp., white), then the number of directed paths of length $2$ from $u$ to $v$ in $Q$ equals the number of red-blue (resp., blue-red) paths of length $2$ from $u$ to $v$ in $G(Q)$.
\end{proof}

\subsection{A condition on $2\times 2$ partitions}\label{subsec:2by2}

In \cite[Definition~1.1]{S}, Stembridge requires an admissible $ADE$ bigraph to have a \emph{$2\times 2$ vertex partition}, that is, a pair of maps $c_\Gamma,c_\Delta:V\to \{0,1\}$ so that
\[(c_\Gamma(u),c_\Delta(u))=(1-c_\Gamma(v),c_\Delta(v))\quad \text{for all $(u,v)\in\Gamma$;}\] 
\[(c_\Gamma(u),c_\Delta(u))=(c_\Gamma(v),1-c_\Delta(v))\quad \text{for all $(u,v)\in\Delta$.}\] 
Let us introduce a temporary terminology and call such bigraphs \emph{admissible $2\times 2$ ADE bigraphs}.

\def\comp{ \operatorname{Comp}(\Gamma)}
\begin{proposition}
 Every admissible $ADE$ bigraph is an admissible $2\times 2$ ADE bigraph.
\end{proposition}
\begin{proof}
 Following \cite[Section 2]{S}, the \emph{component graph} $\comp$ of $\Gamma$ is a simple graph whose vertices are connected components of $\Gamma$, and two of them are connected by an edge in $\comp$ if and only if there is an edge of $\Delta$ between the two corresponding connected components.
 \begin{lemma}[{\cite[Lemma~2.5(b)]{S}}]\label{lemma:acyclic}
  Suppose $(\Gamma, \Delta)$ is a bigraph with commuting adjacency matrices such that $\Delta$ is acyclic. Then the $\comp$ is acyclic.
 \end{lemma}
 Using this lemma, it becomes easy to construct the $2\times 2$ partition $(c_\Gamma,c_\Delta)$. Namely, since the component graph of $\Gamma$ is acyclic, it is bipartite. Let $\tau:V\to \{0,1\}$ be such that $\tau(v)=1$ if and only if $v$ is contained in the connected component of $\Gamma$ that corresponds to a black vertex of $\comp$. Then for any $v\in V$ we can put 
 \begin{eqnarray*}
 c_\Gamma(v)&\equiv& \tau(v)+\e_v\pmod 2;\\
 c_\Delta(v)&=& \tau(v).
 \end{eqnarray*}

 Thus if $(u,v)\in\Gamma$ then $\tau(u)=\tau(v)$ while $\e_u\neq \e_v$. Similarly, if $(u,v)\in\Delta$ then $\tau(u)\neq \tau(v)$ as well as $\e_u\neq \e_v$ so $c_ \Gamma(u)=c_ \Gamma(v)$.

\end{proof}

 \subsection{Reformulation of the dynamics in terms of bigraphs}\label{subsec:bigraph_quiver_translation}

  Let $G=(\Gamma,\Delta)$ be a bipartite bigraph with a vertex set $V$. Then the associated $T$-system for $G$ is defined as follows:
 \begin{eqnarray*}
 T_v(t+1)T_v(t-1)&=&\prod_{(u,v)\in\Gamma} T_u(t)+\prod_{(v,w)\in\Delta} T_w(t);\\
 T_v(\e_v) &=& x_v.
 \end{eqnarray*}
 It is easy to see that this system is equivalent to the $T$-system defined for $Q(G)$. The $Y$- and tropical $T$-systems for $G$ are defined in a similar way: replace $u\to v$ (resp., $v\to w$) by $(u,v)\in\Gamma$ (resp., by $(v,w)\in\Delta$) in the corresponding definition.

\section{Twists} \label{sec:twist}

\def\mupair{\widetilde{\mu}}
A \emph{labeled quiver} is by definition a pair $(Q,\rho)$ where $Q$ is a quiver and $\rho:\VertQ\to\R_{>0}$ is a labeling of its vertices by positive numbers. Let $V$ be any set, and we would like to consider all labeled quivers $(Q,\rho)$ with vertex set $V$. Given a vertex $v\in V$, one can define an operation $\mupair_v$ acting on labeled quivers with vertex set $V$. Namely, $\mupair_v(Q,\rho)=(\mu_v(Q),\rho')$ where $\rho'(u)=\rho(u)$ for all $u\neq v$, and 
\[\rho'(v)\rho(v)=\prod_{u\to v} \rho(u)+\prod_{v\to w}\rho(w).\]
When $Q$ is a bipartite recurrent quiver, we define $\mupair_\circ$ (resp., $\mupair_\bullet$) to be the product of $\mupair_v$ over all white (resp., black) vertices of $Q$. And then the $T$-system can be viewed as a repeated application of $\mupair_\circ\mupair_\bullet$ to some initial value labeling.

\def\spair{s}
\def\QL{Q_\Lambda}
Let $\Lambda$ be a Dynkin diagram, and consider its twist $\Lambda \times \Lambda$. Define the quiver $\QL$ to be $Q(\Lambda \times \Lambda)$. For each node $v$ of $\Lambda$, let $v_1$ and $v_2$ be the two corresponding nodes in $\QL$. Define $\spair_v$ to be the operation on labeled quivers of ``swapping'' the vertices $v_1$ and $v_2$. More precisely, if $(Q,\rho)$ is a labeled quiver with vertex set $\Vert(\QL)$, then $\spair_v(Q,\rho)=(Q',\rho')$ is defined as follows:

\begin{itemize}
 \item $\rho'(u)=\rho(u)$ for all $u\in \Vert(\QL)\setminus\{v_1,v_2\}$;
 \item $\rho'(v_1)=\rho(v_2);\quad \rho'(v_2)=\rho(v_1)$;
 \item $(u\to v_1)$ is an edge of $Q$ if and only if $(u\to v_2)$ is an edge of $Q'$. Similarly for the edges $(u\to v_2),(v_1\to w), (v_2\to w)$. All other edges are the same for $Q$ and $Q'$.
\end{itemize}

Now, since $v_1$ and $v_2$ are not connected by an edge in $\QL$, the mutations $\mupair_{v_1}$ and $\mupair_{v_2}$ applied to $(\QL,\rho)$ commute with each other. This allows us to define the following operation:
\def\mubar{\bar\mu}
\[\mubar_v = \spair_v \circ \mupair_{v_1} \circ \mupair_{v_2}.\] 
Informally, we mutate at $v_1$ and $v_2$, and then swap them in the quiver. 

\begin{prop}
If $\mubar_v(\QL,\rho)=(Q',\rho')$ then $Q'=\QL$. In other words, the operations $\mubar_v$ preserve the quiver $\QL$.
\end{prop}

\begin{proof}
 For each pair of arrows $u \rightarrow v_1 \rightarrow w$ there exists the corresponding pair of arrows $w \rightarrow v_2 \rightarrow u$ and vice versa. Thus, no new arrows are created. The directions of arrows adjacent to $v_1$ and $v_2$ are reversed, however 
 applying $\spair_u$ recovers the original directions. 
\end{proof}

\begin{theorem}
The operations $\mubar_v$ satisfy the relations of the Coxeter group defined by $\Lambda$. Specifically, $(\mubar_u \circ \mubar_w)^3=id$ if $u$ and $w$ are adjacent in $\Lambda$,  and $(\mubar_u \circ \bar \mu_w)^2=id$ otherwise.
\end{theorem}

  \begin{figure}
    \begin{center}
\vspace{-.1in}
\scalebox{0.9}{
\input{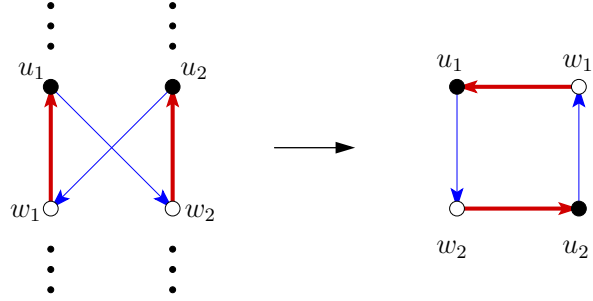} 
}
\vspace{-.1in}
    \end{center} 
    \caption{The local part of $\Lambda \times \Lambda$ as the $Q(A_2 \otimes A_2)$ quiver.}
    \label{fig:zper15}
\end{figure}

\begin{proof}
 The case of $u$ and $w$ not adjacent is trivial. Assume $u$ and $w$ are adjacent. The relation $(\bar \mu_u \circ \bar \mu_w)^3=id$ can be seen to hold via a short direct computation. Alternatively, note that it holds when $u$ and $w$ are the only two vertices of $\Lambda=A_2$: in this case it is just the $A_2 \otimes A_2$ case of Zamolodchikov periodicity, 
 which can be easily verified to hold with principal coefficients.
 By \cite[Theorem 4.6]{FZ2} this implies that the relation holds with arbitrary choice of coefficients, in particular with all the other vertices surrounding $u_1,u_2,w_1,w_2$ in $\QL$.
\end{proof}

\begin{corollary}\label{cor:twists}
 Both $T$-system and tropical $T$-system versions of periodicity hold for twists $\Lambda \times \Lambda$.
\end{corollary}

\begin{proof}
 The time evolution of the $T$-system is the action of 
 \[\mupair_{\circ} \mupair_{\bullet} = \left(\prod_{v \in \Lambda} \spair_v \right) \circ \mubar_\Lambda,\] 
 where 
 \[\mubar_\Lambda = \prod_{\e_v=0} \mubar_v \prod_{\e_v=1} \mubar_v\]

 acts as the Coxeter element in the Coxeter group associated with $\Lambda$. 
 Therefore, $(\mubar_\Lambda)^h = id$, where $h$ is the Coxeter number of $\Lambda$. Taking into account the additional vertex swapping operator $\prod_{v \in \Lambda} \spair_v$ applied at each step that commutes with $\mubar_\Lambda$, we have  
 \[(\mupair_{\circ} \mupair_{\bullet})^{2h} = \left(\prod_{v \in \Lambda} \spair_u  \circ \mubar_\Lambda\right)^{2h}=id,\]
 where the $2$ in the exponent is needed if $h$ is odd and not needed if $h$ is even.
 
 The fact that the same remains true in the tropical $T$-system follows from the general principle that the periodicity of the birational dynamics implies the periodicity of the tropical dynamics, see Proposition \ref{prop:tropical_birational_summary}. 
\end{proof}

\section{Fixed point property $\Longleftrightarrow$ strictly subadditive labeling property} \label{sec:fixsub}

\def\vu{{\nu}}
 Let $Q$ be a bipartite quiver with vertex set $V:=\VertQ$.
 
 Recall from Section \ref{subsec:fixpt}  that a map $\rho:V\to \R_{>1}$ is a \emph{fixed point} if for any $v\in V$, 
 \[\rho(v)^2=\prod_{u\to v}\rho(u)+\prod_{v\to w}\rho(w).\]
 Recall from Section \ref{subsec:subadditive} that a map $\nu:V\to\R_{>0}$ is a \emph{strictly subadditive labeling} if for any $v\in \VertQ$,
 \[2\nu(v)>\sum_{u\to v}\nu(u);\quad 2\nu(v)>\sum_{v\to w}\nu(w).\]
   
\begin{proposition}\label{prop:fixpt_subadditive}
There exists a fixed point for $Q$ if and only if there exists a strictly subadditive labeling for $Q$.
\end{proposition}
\begin{proof}

For any map $\rho:V\to \R_{\geq 1}$, define $Z\rho:V\to \R_{\geq 1}$ by
\[Z\rho(v)=\sqrt{\prod_{u\to v}\rho(u)+\prod_{v\to w}\rho(w)}.\]
Thus we are concerned with the fixed points of a continuous operator $Z:\R_{\geq 1}^{|V|}\to\R_{\geq 1}^{|V|}$.

First, define $\rho_0(v):=1$ for all $v$, then $Z\rho_0(v)\geq \sqrt{2}>\rho_0(v)$ for all $v$. For two maps $\rho',\rho'':V\to \R_{\geq1}$ we write $\rho'> \rho''$ (resp., $\rho'\geq \rho''$) if and only if for all $v$, $\rho'(v)>\rho''(v)$ (resp., $\rho'(v)\geq \rho''(v)$).

Now we can refine Proposition~\ref{prop:fixpt_subadditive} using an intermediate condition:
\begin{lemma}\label{lemma:fixpt_subadditive}
For a bipartite quiver $Q$, the following are equivalent:
\begin{enumerate}
 \item There exists a strictly subadditive labeling for $Q$. \label{item:subadditive}
 \item There exists a map $\rho_1:V\to\R_{>1}$ satisfying $\rho_0< Z\rho_1<\rho_1$.\label{item:rho}
 \item There exists a fixed point for $Q$.\label{item:fixpt}
\end{enumerate}
\end{lemma}
We start with (\ref{item:rho})$\Longrightarrow$(\ref{item:fixpt}). The proof is identical to the proof of \cite[Theorem 2]{Ken} where this theorem is attributed to Tarski. Namely, it is easy to see that if $\rho\geq \rho'$ then $Z\rho\geq Z\rho'$. We have found two maps $\rho_0,\rho_1$ such that $Z\rho_0>\rho_0$ and $Z\rho_1<\rho_1$. Therefore a convex compact set (which is just a parallelotope) 
\[[\rho_0,\rho_1]:=\{\rho:V\to\R_{\geq 1}\mid \rho_0(v)\leq \rho(v)\leq \rho_1(v)\quad \forall\,v\in V\}\] 
is mapped by $Z$ to itself, so (\ref{item:fixpt}) follows immediately from the Brouwer fixed point theorem. Since $Z\rho_0>\rho_0$, none of the coordinates of this fixed point can be equal to $1$ and thus our fixed point is a map $\rho:V\to\R_{>1}$. 

Now we want to show (\ref{item:fixpt})$\Longrightarrow$(\ref{item:subadditive}). Assume $\rho:V\to\R_{>1}$ is a fixed point of $Z$. Define $\nu(v):=\log \rho(v)>0$ and then we have
  \begin{eqnarray*}
 \exp(\nu(v))&=&\rho(v)=Z\rho(v)=\sqrt{\prod_{u\to v}\rho(u)+\prod_{v\to w}\rho(w)}\\
             &>& \max\left\{\sqrt{\prod_{u\to v}\rho(u)},\sqrt{\prod_{v\to w}\rho(w)}\right\},
  \end{eqnarray*}
  so after taking the logarithm of both sides we see that $\nu$ is a strictly subadditive labeling. 
  
The only thing left to show is (\ref{item:subadditive})$\Longrightarrow$(\ref{item:rho}). Assume $\nu: V\to \R_{>0}$ is a strictly subadditive labeling of vertices of $Q$. In other words, for any $v$ we have 
 \[2\nu(v)>\sum_{u\to v}\nu(u),\quad 2\nu(v)>\sum_{v\to w}\nu(w).\]
 Since the number of vertices is finite, there exists a number $q>0$ such that 
 \[2\nu(v)-q>\sum_{u\to v}\nu(u),\quad 2\nu(v)-q>\sum_{v\to w}\nu(w).\]
 
  Let $\alpha>1$ be a big enough real number so that $\alpha^q>2$ and put $\rho_1(v):=\alpha^{\nu(v)}$ for all $v\in V$. Then we have 
  \begin{eqnarray*}
 Z\rho_1(v)&=&\sqrt{\prod_{u\to v}\rho_1(u)+\prod_{v\to w}\rho_1(w)}=\sqrt{\alpha^{\sum_{u\to v}\nu(u)}+\alpha^{\sum_{v\to w}\nu(w)}}\\
      &<&\sqrt{2\alpha^{2\nu(v)-q}}<\alpha^{\nu(v)}=\rho_1(v).   
  \end{eqnarray*}
  Clearly, we have $\rho_0<\rho_1$ and thus $Z\rho_0<Z\rho_1$ so 
  \[\rho_0<Z\rho_0<Z\rho_1<\rho_1,\]
 which finishes the proof of Lemma~\ref{lemma:fixpt_subadditive} and thus the proof of Proposition~\ref{prop:fixpt_subadditive}.
\end{proof}

\section{Strictly subadditive labeling property $\Longleftrightarrow$ admissible $ADE$ bigraph property}  \label{sec:subADE}

\begin{proposition}\label{prop:subadditive_ADE}
 Let $Q$ be a bipartite recurrent quiver. Then $Q$ has a strictly subadditive labeling if and only if $G(Q)$ is an admissible $ADE$ bigraph.
\end{proposition}
\begin{proof}
 We have already shown in Corollary~\ref{cor:recurrent_commuting} that if $Q$ is a bipartite recurrent quiver then the adjacency matrices $A_\Gamma$ and $A_\Delta$ of the associated bipartite bigraph $G(Q)=(\Gamma,\Delta)$ commute with each other. Thus we just need to show that $Q$ has a strictly subadditive labeling if and only if all connected components of $\Gamma$ and of $\Delta$ are $ADE$ Dynkin diagrams. 
 
 One direction follows from Vinberg's characterization (see Theorem~\ref{thm:Vinberg}). If $\nu:\VertQ\to\R_{>0}$ is a strictly subadditive labeling then its restriction to each connected component of $\Gamma$ or $\Delta$ gives a labeling of the vertices of this component satisfying the conditions of Theorem~\ref{thm:Vinberg} and therefore forces this component to be an $ADE$ Dynkin diagram.
 
 \def\v{{ \mathbf{v}}}
 To go in the other direction, we use a construction from the proof of \cite[Lemma~4.5]{S}. In the course of this proof, for any admissible $ADE$ bigraph $(\Gamma,\Delta)$ Stembridge constructs a vector $\v:\VertQ\to\R_{>0}$ such that its restriction onto each connected component of $\Gamma$ or of $\Delta$ is a strictly positive multiple of the dominant eigenvector for (the adjacency matrix of) this connected component. Since the dominant eigenvalue for each connected component of $\Gamma$ (resp., of $\Delta$) is equal to $2\cos(\pi/h)$ (resp., $2\cos(\pi/h')$) where $h,h'$ are common Coxeter numbers provided by Proposition~\ref{prop:coxeter}, we conclude that $\v$ is a \emph{positive common eigenvector for both $A_\Gamma$ and $A_\Delta$ with eigenvalues $2\cos(\pi/h)$ and $2\cos(\pi/h')$ respectively}. By the definition of the adjacency matrix, this can be rewritten as 
 \[2\cos(\pi/h) \v(v)=\sum_{(v,u)\in\Gamma} \v(u)\]
 for all $v\in\VertQ$, and similarly
 \[2\cos(\pi/h') \v(v)=\sum_{(v,w)\in\Delta} \v(w).\]
 Since $2>2\cos(\pi/h)$ and $2>2\cos(\pi/h')$ for any admissible $ADE$ bigraph, it follows that $\v$ is a strictly subadditive labeling for $Q$. 
\end{proof}

\section{Connection between periodicity and tropical periodicity} \label{sec:pertrop}

Let $Q$ be a bipartite recurrent quiver. Then there is the $T$-system and the tropical $T$-system associated with $Q$. For clarity, we are going to call the former the \emph{birational} $T$-system associated with $Q$.

\subsection{Newton polytopes and positivity}
Let $\Z[\x^{\pm1}]$ denote the ring of Laurent polynomials in the variables $\x=\{x_v\}_{v\in\VertQ}$. Since the $T$-system associated with $Q$ is a special case of a cluster algebra, the Laurent phenomenon of Fomin and Zelevinsky \cite{FZ} states that for all $v\in\VertQ$ and $t\in\Z$ with $t+\e_v$ even, the value $T_v(t)$ belongs to $\Z[\x^{\pm1}]$.

\def\Newton{{ \operatorname{Newton}}}
\def\Conv{{ \operatorname{Conv}}}
\def\maxx{{\sup}}
\def\minn{{\inf}}

Let $d:=|\VertQ|$ be the number of vertices of $Q$. In this section, we identify maps $\VertQ\to\Z$ with vectors in $\Z^d$ and for $\alpha\in\Z^d$, we define $\x^ \alpha:=\prod_v x_v^{ \alpha(v)}$.

For any Laurent polynomial $p\in\Z[\x^{\pm1}]$, say,
\[p=\sum_{\alpha\in\Z^d} c_\alpha \x^\alpha,\]
define its \emph{Newton polytope} $\Newton(p)\subset \R^d$ to be 
\[\Newton(p):=\Conv\{\alpha\in\Z^d\mid c_\alpha\neq 0\}.\]
Recall that for two polytopes $P_1,P_2\subset \R^d$, their \emph{Minkowski sum} $P_1+P_2$ is defined as follows:
\[P_1+P_2=\{z_1+z_2\mid z_1\in P_1,\ z_2\in P_2\}.\]
Comparing the definitions, one immediately obtains the following well-known property of Newton polytopes: for $p_1,p_2\in \Z[\x^{\pm1}]$, we have
\[\Newton(p_1p_2)=\Newton(p_1)+\Newton(p_2).\]

\begin{definition}
 A Laurent polynomial 
 \[p=\sum_{\alpha\in\Z^d} c_\alpha \x^\alpha\]
 is called \emph{Newton-positive} if for every $\alpha$ that is a vertex of $\Newton(p)$, we have $c_ \alpha>0$.
\end{definition}

The following lemma is straightforward:
\begin{lemma}\label{lemma:Newton_positive}
Let $p_1,p_2,p\in\Z[\x^{\pm1}]$ be Laurent polynomials and assume $p_1,p_2$ are Newton-positive. Then $p$ is also Newton-positive if it satisfies one of the following:
\begin{itemize}
 \item $p=p_1p_2$, or
 \item $p=p_1/p_2$, or
 \item $p=p_1+p_2$.
\end{itemize}
Moreover, if $p_1$ and $p_2$ are Newton-positive, then  
\[\Newton(p_1+p_2)=\Conv\left(\Newton(p_1)\cup\Newton(p_2)\right).\]
\end{lemma}

\subsection{Degrees and support functions}
For a linear function $\l:\R^d\to \R$ and a polytope $P\subset\R^d$, define $\maxx(\l,P):=\maxx\{\l(z)\mid z\in P\}$. This supremum is actually a maximum since $P$ is compact, but we choose this notation to avoid confusion with the tropical $T$-system relations. 

\begin{lemma}\label{lemma:tropical_birational}
 For a map $\l:\VertQ\to\R$, the tropical $T$-system $\Ttr$ associated with $Q$ is obtained from the birational $T$-system via the following transformation:
 \begin{equation}\label{eq:tropical_birational}
 \Ttr_v(t)=\maxx\left(\l,\Newton(T_v(t))\right) 
 \end{equation}

\end{lemma}
\begin{proof}
Since the Laurent polynomial $x_v$ is Newton-positive for every $v$, it follows from Lemma~\ref{lemma:Newton_positive} that the Laurent polynomial $T_v(t)$ is Newton-positive for all $t\in\Z,v\in\VertQ$. Next, (\ref{eq:tropical_birational}) clearly holds for $t=\e_v$. It remains to note that for any linear function $\l:\R^d\to\R$ and any two polytopes $P_1$ and $P_2$,
 \[\maxx(\l,P_1+P_2)=\maxx(\l,P_1)+\maxx(\l,P_2),\]
 and
 \[\maxx(\l,\Conv(P_1\cup P_2))=\max\left(\maxx(\l,P_1),\maxx(\l,P_2)\right).\]
\end{proof}

\def\maxdeg{{ \operatorname{deg}_{\max}}}
\def\mindeg{{ \operatorname{deg}_{\min}}}

\begin{definition}
 For $p\in\Z[\x^{\pm1}]$ and $v\in\VertQ$, we define $\maxdeg(v,p) \in \Z$ (resp., $\mindeg(v,p)\in\Z$) to be the maximal (resp., minimal) degree of $x_v$ in $p$ viewed as an element of $\Z[x_v^{\pm1}]$ with all the other indeterminates $\{x_u\}_{u\neq v}$ regarded as constants.
\end{definition}
 For $u\in\VertQ$, set $\delta_u:\VertQ\to\R$ to be 
 \[\delta_u(v):=\begin{cases}
                1, &\text{if $u=v$;}\\
                0, &\text{otherwise.}
               \end{cases}\]

As always, we extend this map by linearity to a map $\delta_u:\Z^d\to\R$. Then $\maxdeg(u,p)$ and $\mindeg(u,p)$ can be expressed in terms of the Newton polytope of $p$ as follows:
 \begin{eqnarray*}
 \maxdeg(u,p)&=&\maxx(\delta_u,\Newton(p));\\
\mindeg(u,p)&=&\minn(\delta_u,\Newton(p)):=-\maxx(-\delta_u,\Newton(p)).
 \end{eqnarray*}

\begin{corollary}
We have 
 \[\t^{ \delta_u}_v(t)=\maxdeg(u,T_v(t))\]
 for all $u,v\in\VertQ$ and all $t\in\Z$ with $t+\e_v$ even.
\end{corollary}
That is, the tropical $T$-system with initial condition $\delta_u$ controls the maximal degrees of $x_u$ in Laurent polynomials appearing as entries of the birational $T$-system. The following proposition shows that it controls the minimal degrees of $x_u$ as well, even though we still take the maximum in the recurrence for $\Ttr_v(t)$:
\begin{proposition}\label{prop:maxdeg_mindeg}
For all $t\in\Z$ and $u,v\in \VertQ$ such that $t+\e_v$ is even,
\begin{itemize}
 \item  If $\e_u=0$ then 
 \begin{equation}\label{eq:eps0}
 -\mindeg(u,T_v(t))=\maxdeg(u,T_v(t-2))=\t^{\delta_u}_v(t-2);
 \end{equation} 
 \item  If $\e_u=1$ then 
 \begin{equation}\label{eq:eps1}
 -\mindeg(u,T_v(t))=\maxdeg(u,T_v(t+2))=\t^{\delta_u}_v(t+2)
 \end{equation} 
 \end{itemize}
\end{proposition}
\begin{proof}
 The integers $-\mindeg(u,T_v(t))$ clearly satisfy the same recurrence as the integers $\maxdeg(u,T_v(t))$. 
 Therefore it only remains to check the initial case $t=\e_v$. Assume for example that $\e_u=0$. Initially, 
 \[\maxdeg(u,T_v(\e_v))=\delta_u(v)\] 
 and 
 \[-\mindeg(u,T_v(\e_v))=-\delta_u(v).\] 
 Let us now look at $t=2$, thus $T_v(t)$ is defined only for $\e_v=0$. Since $\e_u=0$ as well, $v$ and $u$ have to be of the same color, so they are not connected and thus
 \[\maxdeg(u,T_v(2))=\max(0,0)-\maxdeg(u,T_v(0))=-\delta_u(v).\]
 For $t=3$, $v$ and $u$ have to be of different color. So they are either not connected in which case we will get a zero again, or they are connected in which case we will get 
 \[\maxdeg(u,T_v(3))=\max(-1,0)-\maxdeg(u,T_v(1))=0.\]
 To sum up,
 \[\maxdeg(u,T_v(2+\e_v))=-\delta_u(v)=-\mindeg(u,T_v(\e_v)).\]
 Since both sides of (\ref{eq:eps0}) agree on the initial conditions and satisfy the same recurrence, they are equal. The case $\e_u=1$ in the proof of (\ref{eq:eps1}) is treated similarly. 
\end{proof}

\begin{example}
Let $Q,a,b,c,\e$ be as in Example~\ref{example:A3_A1}. In order to illustrate Proposition~\ref{prop:maxdeg_mindeg}, we give the values of $\t^{\delta_u}_v(t)$ for all $u,v\in\VertQ$ and all $t=0,\dots,13$ with even $t+\e_v$ in Table~\ref{tab:tropical_delta}.

\begin{table}
\centering
\begin{tabular}{|c|ccc|ccc|ccc|}\hline
$u$                         &  $a$  &       &        &        & $b$    &        &        &        & $c$    \\\hline         
\backslashbox{$t$}{$v$}     &  $a$  & $b$   &  $c$   &  $a$   & $b$    &  $c$   &  $a$   & $b$    &  $c$   \\\hline
$13$                        &       & $ 0$  &        &        & $ 1$   &        &        & $ 0$   &         \\\hline
$12$                        &  $1$  &       & $ 0$   &  $ 0$  &        & $ 0$   &  $ 0$  &        &  $ 1$   \\\hline
$11$                        &       & $ 1$  &        &        & $-1$   &        &        & $ 1$   &         \\\hline
$10$                        &  $0$  &       & $ 1$   &  $ 0$  &        & $ 0$   &  $ 1$  &        &  $ 0$   \\\hline
$ 9$                        &       & $ 0$  &        &        & $ 1$   &        &        & $ 0$   &         \\\hline
$ 8$                        &  $0$  &       & $-1$   &  $ 1$  &        & $ 1$   &  $-1$  &        &  $ 0$   \\\hline
$ 7$                        &       & $ 0$  &        &        & $ 1$   &        &        & $ 0$   &         \\\hline
$ 6$                        &  $0$  &       & $ 1$   &  $ 0$  &        & $ 0$   &  $ 1$  &        &  $ 0$   \\\hline
$ 5$                        &       & $ 1$  &        &        & $-1$   &        &        & $ 1$   &         \\\hline
$ 4$                        &  $1$  &       & $ 0$   &  $ 0$  &        & $ 0$   &  $ 0$  &        &  $ 1$   \\\hline
$ 3$                        &       & $ 0$  &        &        & $ 1$   &        &        & $ 0$   &         \\\hline
$ 2$                        &  $-1$ &       & $ 0$   &  $ 1$  &        & $ 1$   &  $ 0$  &        &  $-1$   \\\hline
$ 1$                        &       & $ 0$  &        &        & $ 1$   &        &        & $ 0$   &         \\\hline
$ 0$                        &  $1$  &       & $ 0$   &  $ 0$  &        & $ 0$   &  $ 0$  &        &  $ 1$   \\\hline
\end{tabular}
\caption{\label{tab:tropical_delta} The values of $\t^{\delta_u}_v(t)$}
\end{table}

For example, for $u=a$ and $t=2$, we see that 
\[\t^{\delta_u}_v(2+\e_v)=-\t^{\delta_u}_v(\e_v)=-\mindeg(u,\Newton(x_v)).\] 
For $u=b$ and $t=-2$,
\[\t^{\delta_u}_v(-2+\e_v)=\t^{\delta_u}_v(10+\e_v)=-\t^{\delta_u}_v(\e_v)=-\mindeg(u,\Newton(x_v)).\]
\end{example}

We summarize the results of this section as follows:

\begin{proposition}\label{prop:tropical_birational_summary}
 For an integer $N$ and a bipartite recurrent quiver $Q$, the following are equivalent:
 \begin{enumerate}
 \item \label{item:weak_tropical} $\t^{\delta_u}_v(\e_v+2N)=\t^{\delta_u}_v(\e_v)$ for all $u,v\in\VertQ$;
  \item \label{item:strong_tropical} $\Ttr_v(t+2N)=\Ttr_v(t)$ for all $\l:\VertQ\to\R$ and $v\in\VertQ$, $t\in\Z$ with $t+\e_v$ even;
  \item \label{item:birational} There exists a map $c:\VertQ\to\R_{>0}$ such that $T_v(t+2N)=c(v)T_v(t)$ for all $v\in\VertQ,t\in\Z$ with $t+\e_v$ even.
 \end{enumerate}
 In any of these cases, if $Q$ has the fixed point property then $c(v)=1$ for all $v$.
\end{proposition}
\begin{proof}
 Note that (\ref{item:strong_tropical})$\Longrightarrow$(\ref{item:weak_tropical}) is trivial. Also, (\ref{item:birational})$\Longrightarrow$(\ref{item:strong_tropical}) follows directly from Lemma~\ref{lemma:tropical_birational}.
 
 To show (\ref{item:weak_tropical})$\Longrightarrow$(\ref{item:birational}), note that it is enough to consider $t=\e_v$, because if $T_v(2N+\e_v)=c(v)T_v(\e_v)$ for all $v$ then we can get the case of arbitrary $t$ by a substitution $x_v:=T_v(t+\e_v)$ for all $v$. By Proposition~\ref{prop:maxdeg_mindeg}, (\ref{item:weak_tropical}) implies that for every $u,v\in\VertQ$, 
 \[\maxdeg(u,T_v(2N+\e_v))=\maxdeg(u,T_v(\e_v));\]
 \[\mindeg(u,T_v(2N+\e_v))=\mindeg(u,T_v(\e_v)).\]
 Therefore the Laurent polynomials $T_v(2N+\e_v)$ and $T_v(\e_v)$ differ by a scalar multiple $c(v)$. But since they are both Newton-positive, $c(v)$ is necessarily positive, so (\ref{item:birational}) follows.
 
 Now, if $\rho:\VertQ\to\R_{>1}$ is a fixed point, then substituting $x_v:=\rho(v)$ for all $v$ into $T_v(t+2N)=c(v)T_v(t)$ yields $\rho(v)=c(v)\rho(v)$ and thus $c(v)=1$.
\end{proof}

Proposition~\ref{prop:tropical_birational_summary} is important both for the rest of the proof and for computer checks: it reduces the problem of checking periodicity of the birational $T$-system to the problem of checking periodicity of $\t^{\delta_u}_v$ for every $u,v\in\VertQ$. Unlike the birational $T$-system where the values of $T_v(t)$ are multivariate Laurent polynomials, here we are dealing with small integers. As a result, this dramatically reduces the complexity of the code and the amount of time needed to check the exceptional cases (see Remark~\ref{remark:computation}).

\section{Tropical periodicity implies the strictly subadditive labeling property} \label{sec:tropsub}

\begin{proposition}\label{prop:tropical_subadditive}
 Let $Q$ be a bipartite recurrent quiver and assume that the tropical $T$-system associated with $Q$ is periodic. Then $Q$ has the strictly subadditive labeling property.
\end{proposition}
\begin{proof}
 We let $N$ be an integer such that for all $u,v\in\VertQ$, $\t^{\delta_u}_v(\e_v+2N)=\t^{\delta_u}_v(\e_v)$ (by Proposition~\ref{prop:tropical_birational_summary}, this is a special case of tropical periodicity).  Now, we want to find a strictly subadditive labeling $\nu:\VertQ\to\R_{>0}$. In order to do it, we construct certain intermediate maps $b^u:\VertQ\to\R$. Following the proof of \cite[Theorem 9.1]{Reu}, we let $b^u$ be just the sum of $\maxdeg(u,T_v(t))$ over the period:
 \[b^u(v):=\sum_{j=0}^{N-1} \t^{\delta_u}_v(2j+\e_v).\]
 
 First, we want to show that $b^u(v)\geq 0$ for all $u,v\in\VertQ$. By Proposition~\ref{prop:maxdeg_mindeg}, the average of $\maxdeg(u,T_v(t))$ over $t$ equals the average of $-\mindeg(u,T_v(t))$. Since $\maxdeg(u,T_v(t))\geq \mindeg(u,T_v(t))$, we get that $b^u(v)\geq -b^u(v)$. This implies that $b^u(v)\geq 0$ for all $u,v\in\VertQ$. 
 
 In addition, if $u$ and $v$ are connected by an edge, and $\e_u=1,\e_v=0$ then $\maxdeg(u,T_v(2))=1>0=\mindeg(u,T_v(2))$ and in this case $b^u(v)>-b^u(v)$. Similarly, if $\e_u=0,\e_v=1$ then $\maxdeg(u,T_v(-1))=1>0=\mindeg(u,T_v(-1))$. Hence \emph{$b^u(v)>0$ if $u$ and $v$ are neighbors in $Q$}.
 
 Second, we claim that the labeling $b^u$ is subadditive (even though it may or may not be strictly subadditive):
 \begin{equation}\label{eq:subadditive}
 b^u(v)\geq \max\left(\sum_{w\to v} b^u(w),\sum_{v\to w} b^u(w)\right). 
 \end{equation}

 Indeed, using the periodicity of $\t^{\delta_u}_v$, we can write
 \begin{equation*}
 \begin{split}
  2b^u(v)&=\sum_{j=0}^{N-1} \t^{\delta_u}_v(2j+\e_v)+\sum_{j=1}^{N} \t^{\delta_u}_v(2j+\e_v)\\
  &= \sum_{j=0}^{N-1} (\t^{\delta_u}_v(2j+\e_v)+\t^{\delta_u}_v(2j+2+\e_v)),
  \end{split}
  \end{equation*}
  which by~\eqref{eq:tropical} equals to
  \begin{equation*}
  \begin{split}
  &=\sum_{j=0}^{N-1} \max\left(\sum_{w\to v} \t^{\delta_u}_w(2j+1+\e_v),\sum_{v\to w} \t^{\delta_u}_w(2j+1+\e_v)\right)\\
  &\geq\max\left(\sum_{j=0}^{N-1} \sum_{w\to v} \t^{\delta_u}_w(2j+1+\e_v),\sum_{j=0}^{N-1}\sum_{v\to w} \t^{\delta_u}_w(2j+1+\e_v)\right)\\
  &=\max\left(\sum_{w\to v} b^u(w),\sum_{v\to w} b^u(w)\right).
  \end{split}
 \end{equation*}

 Thus $b^u(v)$ is subadditive. The only way (\ref{eq:subadditive}) can be an equality for some vertex $v$ is when either of the following holds:
 \begin{itemize}
  \item \label{item:j1} for all $j\in\{0,1,\dots,N-1\}$,
  \begin{equation}\label{eq:all_j_1}
\sum_{w\to v} \t^{\delta_u}_w(2j+1+\e_v)\geq\sum_{v\to w} \t^{\delta_u}_w(2j+1+\e_v).   
  \end{equation}
  \item \label{item:j2} for all $j\in\{0,1,\dots,N-1\}$,
   \begin{equation}\label{eq:all_j_2}
  \sum_{w\to v} \t^{\delta_u}_w(2j+1+\e_v)\leq\sum_{v\to w} \t^{\delta_u}_w(2j+1+\e_v).
  \end{equation}
 \end{itemize}
 As we note in the proof of Proposition~\ref{prop:maxdeg_mindeg}, there are two integers $j_1$ and $j_2$ such that for all $w$,
 \[\t^{\delta_u}_w(2j_1+1+\e_v)=\delta_u(w);\]
 \[\t^{\delta_u}_w(2j_2+1+\e_v)=-\delta_u(w).\]
 Thus, if $v$ is a neighbor of $u$ then neither (\ref{eq:all_j_1}) nor (\ref{eq:all_j_2}) happens for all values of $j$ simultaneously. It follows that \emph{the inequality (\ref{eq:subadditive}) is strict when $v$ is a neighbor of $u$}. This allows us to define $\nu$ as a sum of all $b^u$'s:
 \[\nu(v):=\sum_{u\in\VertQ} b^u(v).\]
 Clearly, $\nu$ inherits subadditivity from $b^u$, and since each vertex $v$ is a neighbor of some vertex $u$, $\nu$ is strict at every vertex. By the same reason, $\nu(v)>0$ for every $v\in\VertQ$.  
\end{proof}
\begin{remark}
 The proof of \cite[Theorem 9.1]{Reu} can be applied directly to show that the periodicity of the \emph{birational} $T$-system with all $x_v$'s set to be equal to $1$ implies the strictly subadditive labeling property. More specifically, one can set $\nu(v)$ to be the logarithm of 
 \[\prod_{j=0}^{N-1} T_v(\e_v+2j)\mid_{\x:=1}.\]
 It follows that such a labeling $\nu$ is strictly subadditive. However, there are bipartite quivers for which the birational and tropical $T$-systems are periodic only for \emph{some} initial conditions. For example, if $Q$ is a quiver of type $A_2\otimes A_2$ with three copies of each arrow, one can find a ``fixed point'' $\rho:\VertQ\to\R_{>0}$ defined by $\rho(v)=1/2$ for all $v\in\VertQ$. Similarly, there is a ``strictly subadditive labeling'' $\nu:\VertQ\to\R$ defined by $\nu(v)=-1$ for all $v\in\VertQ$. Thus we see that the requirements that the fixed point must take values in $\R_{>1}$ and that the strictly subadditive labeling must take values in $\R_{>0}$ cannot be dropped. If one instead takes just two copies of each arrow in $A_2\otimes A_2$ and sets $\l(v):=1$ for all $v\in\VertQ$, one gets a periodic tropical $T$-system with $\t^\l_v(t)=1$ for all $v\in\VertQ$, $t\in\Z$ with $t+\e_v$ even. Lots of other examples with the tropical $T$-system being periodic for some $\l$ can be constructed using the duality in Section~\ref{subsec:involutions}.
\end{remark}

\section{Example: proof of the periodicity in the $A_n\otimes A_m$ case}\label{sec:AA}
Since $A_n\otimes A_m$ is an admissible $ADE$ bigraph, we know by Propositions~\ref{prop:fixpt_subadditive} and~\ref{prop:subadditive_ADE} that there is a strictly subadditive labeling $\nu:\VertQ\to\R_{>0}$ as well as a fixed point $\rho:\VertQ\to\R_{>1}$. Here $Q=Q(A_n\otimes A_m)$. By Proposition~\ref{prop:tropical_birational_summary}, we know therefore that the birational $T$-system satisfies
\begin{equation}\label{eq:birational_period}
T_v(t)=T_v(t+2N) 
\end{equation}

if and only if the tropical $T$-system satisfies 
\begin{equation}\label{eq:tropical_period}
\t^{\delta_u}_v(\e_v+2N)=\t^{\delta_u}_v(\e_v) 
\end{equation} 
for all $u,v\in\VertQ$. Thus, in order to prove (\ref{eq:birational_period}) for $Q$, we only need to check (\ref{eq:tropical_period}) for all $u,v$. 

\def\abar{{\bar a}}
\def\bbar{{\bar b}}

First, let us introduce some notation. The vertices of $Q$ can be naturally labeled by pairs of numbers $(a,b)$ with $a=1,2,\dots,n$ and $b=1,2,\dots,m$. For each such pair, we denote the corresponding vertex of $Q$ by $v(a,b)$. We define $\abar:=n+1-a$ and $\bbar:=n+1-b$ to be the coordinates of the opposite vertex of $Q$. For an illustration of the below argument, see Example~\ref{example:A3_A1} for $n=3,m=1$.

\def\SW{{ \mathrm{SW}}}
\def\SE{{ \mathrm{SE}}}
\def\NW{{ \mathrm{NW}}}
\def\NE{{ \mathrm{NE}}}

For the rest of this section we assume that $\t^\l_v(t)$ is defined for all $v\in\VertQ,t\in\Z$, and when $t+\e_v$ is odd we just put $\t^\l_v(t):=\t^\l_v(t-1)$.
\begin{lemma}
 Let $u:=v(a,b)$. Assume that $\e_u=1$. Define the following four sequences of numbers whose entries are indexed by $t=1,2,\dots,n+m$:
 \begin{eqnarray*}
 r_\SE&=&(a+b,a+b+1,\dots,n+m-1,n+m,n+m-1,\dots,\abar+\bbar);\\
 r_\NW&=&(a+b,a+b-1,\dots,3,2,3,\dots,\abar+\bbar);\\
 r_\SW&=&(b-a,b-a-1,\dots,2-n,1-n,2-n,\dots,\bbar-\abar);\\
 r_\NE&=&(b-a,b-a+1,\dots,m-2,m-1,m-2,\dots,\bbar-\abar).  
 \end{eqnarray*}

 Then for $t=1,2,\dots,n+m$,
 \[\t^{\delta_u}_{v(i,j)}(t)=\begin{cases}
                       1, &\text{if $r_\NW(t)\leq i+j\leq r_\SE(t)$ and $r_\SW(t)\leq j-i\leq r_\NE(t)$;}\\
                       0, &\text{otherwise.}
                      \end{cases}\]
\end{lemma}
\begin{proof}
 Follows by induction from the defining relations of the tropical $T$-system.
\end{proof}
\def\weirdstar{{\ast}}
Thus, for the case $\e_u=1$, we get the following dynamics. For $t=0,1$, $\t^{\delta_u}_\weirdstar(t)=\delta_u$ viewed as a map $v\mapsto \t^{\delta_u}_v(t)$. Next, for $t=1,2,\dots,n+m$, $\t^{\delta_u}_\weirdstar(t)$ is filled with ones inside a certain rectangle bounded by four moving diagonals and zero outside. Each of these diagonals moves away from the vertex $v(a,b)$ until it hits the corner of the rectangle, and it starts moving back immediately after that until the four diagonals meet at the vertex $v(\abar,\bbar)$. Next, 
\[\t^{\delta_u}_\weirdstar(n+m+1)=\t^{\delta_u}_\weirdstar(n+m+2)=\t^{\delta_u}_\weirdstar(n+m+3)=-\delta_{v(\abar,\bbar)}\]
since we take the maximum in the definition of the tropical $T$-system and thus the negative value does not propagate. For $t=n+m+3,\dots,2(n+m+1)$ we see our process in reverse, that is, 
\[\t^{\delta_{v(a,b)}}_\weirdstar(t+n+m+2)=\t^{\delta_{v(\abar,\bbar)}}_\weirdstar(t)\]
for all $t\in\Z$, just because they coincide for $t=0,1$. Thus,
\[\t^{\delta_u}_\weirdstar(t+2(n+m+2))=\t^{\delta_u}_\weirdstar(t)\]
for all $u\in\VertQ$.  For the case $\e_u=0$, we start by mutating the value at $u$ into a $-1$, so time in this case goes in the reverse direction. Note that the Coxeter numbers for $Q$ from Proposition~\ref{prop:coxeter} are $h=n+1$ and $h'=m+1$ so our tropical $T$-system satisfies 
\[\t^{\delta_u}_v(\e_v+2(h+h'))=\t^{\delta_u}_v(\e_v).\]
Since a fixed point exists for $Q$, we deduce from Proposition~\ref{prop:tropical_birational_summary} that the same relation is satisfied by the birational $T$-system. 

\section{Admissible $ADE$ bigraphs have the tropical periodicity property} \label{sec:ADEtrop}
Stembridge's classification~\cite{S} contains five infinite families, one of them being tensor products of pairs of Dynkin diagrams, for which the periodicity was shown by Keller in~\cite{K}. The second family, twists $\Lambda\times\Lambda$, give rise to periodic $T$-systems by Corollary~\ref{cor:twists}. A crucial observation is that the remaining three families possess rich symmetries. In particular, recall from Section~\ref{subsec:classification} that each of them was constructed using the involutions described in Section~\ref{sec:ADEDynkin}. We are going to extend these involutions to Stembridge's bigraphs and show that the operation of taking the dual with respect to such an involution ``commutes'' with applying tropical mutations.

\def\Fix{{ \operatorname{Fix}}}
\def\iinv{{\i}}
\def\jinv{{\mathbf{j}}}
\def\gg{{\bar\Gamma}}
\def\dd{{\bar\Delta}}

\subsection{Symmetric labeled bigraphs}\label{subsec:involutions}
The right combinatorial object for us turns out to be rather involved. 

\begin{definition}
 A \emph{symmetric labeled bigraph} $B:=(G,\iinv,V_+,V_0,V_-,\rho,X)$ consists of the following data:
 \begin{enumerate}
  \item $G=(\Gamma,\Delta)$ is a bipartite bigraph with vertex set $V$.
  \item $\iinv:V\to V$ is an automorphism of $G$ of order $2$ that preserves the colors of vertices and edges. 
  \item $V_0=\Fix(\iinv)$ is the set of vertices fixed by $\iinv$. 
  \item We have a partition 
  \[V=V_+\sqcup V_0\sqcup V_-\]
  such that whenever $u\neq v,\ u=\iinv(v)$ then either $u\in V_+,v\in V_-$ or $u\in V_-$ and $v\in V_+$. 
  \item $\rho:V\to \R$ is an \emph{$\iinv$-symmetric} labeling of vertices of $G$. That is, for any $v\in V$ we have 
  \[\rho(\iinv(v))=\rho(v).\]
  \item Finally, $X$ is an arbitrary set of edges $(u,v)$ that are \emph{fixed} by $\iinv$. Since $\iinv$ preserves the color of each vertex, an edge $(u,v)$ is fixed by $\iinv$ if and only if $u=\iinv(u)$ and $v=\iinv(v)$. 
 \end{enumerate}
\end{definition}
{\sloppy
Assume that $V_+=\{v_1^+,\dots,v_k^+\}$, $V_-=\{v_1^-,\dots,v_k^-\}$, and $V_0=\{v_{k+1},\dots,v_n\}$ with $\iinv(v_j^+)=v_j^-$ for all $j=1,\dots, k$ and $\iinv(v_j)=v_j$ for all $j=k+1,\dots, n$. Given a symmetric labeled bigraph $B$, we would like to construct its \emph{dual} symmetric labeled bigraph $B^*=(G^*,\iinv^*,U_+,U_0,U_-,\rho^*,X^*)$. This dual symmetric labeled bigraph is defined via the following rules. First, for each vertex $v_j$ of $B$ that is fixed by $\iinv$ we introduce two vertices $u_j^+$ and $u_j^-$ of $B^*$, and for every two vertices $v_j^+$ and $v_j^-$ of $B$ that are mapped to each other by $\iinv$, 
we introduce just one vertex $u_j$ of $B$. Thus
\[U_0=\{u_1,\dots,u_k\};\quad U_+=\{u_{k+1}^+,\dots,u_{n}^+\};\quad  U_-=\{u_{k+1}^-,\dots,u_{n}^-\}.\]
The involution $\iinv^*$ is defined by $\iinv^*(u_j)=u_j$ for $j=1,\dots, k$ and $\iinv^*(u_j^+)=u_j^-$ for $j=k+1,\dots,n$. Define $U:=U_+\sqcup U_0\sqcup U_-$ to be the set of vertices of $G^*$. The labeling $\rho^*$ is such that
\[\rho^*(u_j)=\sqrt2 \rho(v_j^+)=\sqrt2 \rho(v_j^-),\text{ for $j=1,\dots,k$;}\] \[\rho^*(u_j^+)=\rho^*(u_j^-)=\rho(v_j)/\sqrt2,\text{ for $j=k+1,\dots,n$}.\]
Finally, let us define the bigraph $G^*=(\Gamma^*,\Delta^*)$. Note that since $\iinv$ preserves the color of the vertices, $v_j^+$ and $v_j^-$ are never connected by an edge in $G$. }

We construct $\Gamma^*$ and $X^*\cap \Gamma^*$ by the following set of rules (which is $180^\circ$-symmetric):
\begin{eqnarray*}
 (v_i,v_j)\in \Gamma,\ (v_i,v_j)\in X        &\Longleftrightarrow& (u_i^+,u_j^-),(u_i^-,u_j^+)\in \Gamma^*\\
 (v_i,v_j)\in \Gamma,\ (v_i,v_j)\not\in X    &\Longleftrightarrow& (u_i^+,u_j^+),(u_i^-,u_j^-)\in \Gamma^*\\
 (v_i^+,v_j),(v_i^-,v_j)\in \Gamma           &\Longleftrightarrow& (u_i^+,u_j),(u_i^-,u_j)\in \Gamma^*\\
 (v_i^+,v_j^+),(v_i^-,v_j^-)\in \Gamma       &\Longleftrightarrow& (u_i,u_j)\in \Gamma^*,\ (u_i,u_j)\not\in X^* \\
 (v_i^+,v_j^-),(v_i^-,v_j^+)\in \Gamma       &\Longleftrightarrow& (u_i,u_j)\in \Gamma^*,\ (u_i,u_j)\in X^* .
\end{eqnarray*}
\def\B{{b}}
The definition of $\Delta^*$ and $X^*\cap\Delta^*$ is completely analogous. Thus we have defined $B^*$, the dual of a symmetric labeled bigraph $B$. It immediately follows from the construction that $B^*$ is also a symmetric labeled bigraph and that $(B^*)^*=B$. We define a \emph{symmetric unlabeled bigraph} to be a six-tuple $\B=(G,\iinv,V_+,V_0,V_-,X)$ satisfying the conditions above (so we omit the labeling $\rho$). The construction of the dual restricts to symmetric unlabeled bigraphs with the same property that $(\B^*)^*=\B$. 

\begin{example}
 Let $\B=(A_{2n-1},\i,V_+,V_0,V_-,\emptyset)$ be a symmetric unlabeled bigraph of type $A_{2n-1}$. That is, $\Delta=\emptyset$ and $\i$ is the involution described in Section~\ref{sec:ADEDynkin}. Then from the above discussion it follows that $\B^*=(D_{n+1},\i,U_+,U_0,U_-,\emptyset)$ is a symmetric unlabeled bigraph of type $D_{n+1}$, see Figure~\ref{fig:zper16}. Here, 
 \[|V_0|=|U_+|=|U_-|=1;\quad |V_+|=|V_-|=|U_0|=n-1.\]
   \begin{figure}
    \begin{center}
\vspace{-.1in}
\scalebox{1}{
\input{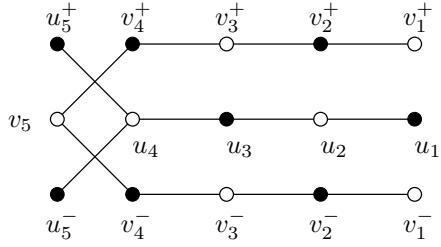} 
}
\vspace{-.1in}
    \end{center} 
    \caption{An example for $n=5$.}
    \label{fig:zper16}
\end{figure}
  More elaborate examples will be considered in Section~\ref{subsec:reduction}.
\end{example}

\begin{proposition}\label{prop:duality}
 Let $B:=(G,\iinv,V_+,V_0,V_-,\rho,X)$ be a  symmetric labeled bigraph, let $B^*=(G^*,\iinv^*,U_+,U_0,U_-,\rho^*,X^*)$ be its dual. Let $v_i\in V_0$ be a vertex fixed by $\iinv$. Define $\l:V\to\R$ and $\l^*:U\to\R$ so that 
 \[\l(v_j)=\rho(v_j);\ \l(v_j^\pm)=\rho(v_j^\pm);\ \l^*(u_j)=\rho^*(u_j);\ \l^*(u_j^\pm)=\rho^*(u_j^\pm)\]
 for all $j\neq i$ and 
 \begin{eqnarray}
 \label{eq:lambda_def}\l(v_i)+\rho(v_i)&=&\max\left(\sum_{(w,v_i)\in\Gamma}\rho(w),\sum_{(v_i,w)\in\Delta}\rho(w)\right);\\
 \label{eq:lambda_def_star}\l^*(u_i^\pm)+\rho^*(u_i^\pm)&=&\max\left(\sum_{(w,u_i^\pm)\in\Gamma^*}\rho^*(w),\sum_{(u_i^\pm,w)\in\Delta^*}\rho^*(w)\right).
 \end{eqnarray}
In other words, $\l$ (resp., $\l^*$) is a result of a tropical mutation at vertex $v_i$ (resp., at vertices $u_i^+,u_i^-$). Then the bigraphs $(G,\iinv,V_+,V_0,V_-,\l,X)$ and $(G^*,\iinv^*,U_+,U_0,U_-,\l^*,X^*)$ are also dual to each other.
\end{proposition}
\begin{proof}
 We only need to show that $\l(v_i)=\sqrt2 \l^*(u_i^+)$. We already know that the same holds for $\rho(v_i)$, so it remains to show that the right hand side of (\ref{eq:lambda_def}) is  $\sqrt2$ times greater than the right hand side of (\ref{eq:lambda_def_star}). Consider a vertex $w$ such that $(w,v_i)\in\Gamma$. If $w\in V_0$ then it means that there are two vertices $w^+$ and $w^-$ in $G^*$ such that either $(w^+,u_i^+)\in\Gamma^*$ or $(w^-,u_i^+)\in\Gamma^*$, but not both. Thus the contribution of $w$ to the sum is $\rho(w)$ while the combined contribution of $w^+$ and $w^-$ is $\rho(w^+)=\rho(w)/\sqrt2$. Now, if $w\in V_+$ or $w\in V_-$ then $\rho(w)=\rho(w^*)/\sqrt2$ where $w^*$ is the vertex of $U_0$ that corresponds to both $w$ and $\iinv(w)$. Since $\rho(w^*)=\sqrt2\rho(w)$, we are done again.
\end{proof}

Since the tropical $T$-system consists of applying tropical mutations to the vertices of $G$, we get an immediate corollary:
\begin{corollary}\label{cor:duality}
 Let $B:=(G,\iinv,V_+,V_0,V_-,\rho,X)$ be a  symmetric labeled bigraph, let $B^*=(G^*,\iinv^*,U_+,U_0,U_-,\rho^*,X^*)$ be its dual. Assume that both $G$ and $G^*$ are recurrent. Then 
 \[\t^{\rho}_v(t+2N)=\t^{\rho}_v(t)\ \forall\,v\in V\Longleftrightarrow \t^{\rho^*}_u(t+2N)=\t^{\rho^*}_u(t)\ \forall\,u\in U,\]
 where $\t^\rho$ is the tropical $T$-system for $G$ and $\t^{\rho^*}$ is the tropical $T$-system for $G^*$. 
\end{corollary}
\begin{proof}
Using Proposition~\ref{prop:duality} as the induction step, we see that the symmetric labeled bigraphs 
\[(G,\iinv,V_+,V_0,V_-,\t^{\rho}_\weirdstar(t),X)\quad \text{and}\quad (G^*,\iinv^*,U_+,U_0,U_-,\t^{\rho^*}_\weirdstar(t),X^*)\] 
are dual to each other.
\end{proof}

Since we would like to show that $N$ can be set to $h(G)+h'(G)$, and we already know that for tensor products, the following proposition will be useful later.
\begin{proposition}\label{prop:coxeter_symmetric}
 Let $B:=(G,\iinv,V_+,V_0,V_-,\rho,X)$ be a  symmetric labeled bigraph, let $B^*=(G^*,\iinv^*,U_+,U_0,U_-,\rho^*,X^*)$ be its dual. If $G$ and $G^*$ are admissible $ADE$ bigraphs then the corresponding Coxeter numbers of $G$ and $G^*$ from Proposition~\ref{prop:coxeter} are equal:
 \[h(G)=h(G^*);\quad h'(G)=h'(G^*).\]
\end{proposition}
\begin{proof}
 The involution $\iinv$ either maps each connected component $C$ of $\Gamma$ isomorphically onto another connected component of $\Gamma$ (then the corresponding connected component of $\Gamma^*$ will be isomorphic to them as well) or it maps $C$ onto itself via one of the maps described in Section~\ref{sec:ADEDynkin}. 
\end{proof}

\subsection{Reduction to tensor products}\label{subsec:reduction}
We roughly describe our strategy of using Corollary~\ref{cor:duality}. Consider a bigraph $G$ belonging to one of the three remaining infinite families. We will find two involutions $\iinv$ and $\jinv$ on $G$ such that $\Fix(\iinv)\cup \Fix(\jinv)=\Vert(G)$. Now, if $u\in \Fix(\iinv)$ is fixed by $\iinv$ then $\delta_u$ is an $\iinv$-symmetric labeling. Then we will find a set $X_\iinv$ of edges that makes $(G,\iinv)$ into a symmetric unlabeled bigraph whose dual is a tensor product of Dynkin diagrams. Since in this case the tropical $T$-system has already been shown to be periodic in~\cite{K}, we can deduce that $\t^{\delta_u}_v(t+2N)=\t^{\delta_u}_v(t)$, where $2N$ is the period of the corresponding $T$-system and $u\in\Fix(\iinv)$. Applying the same construction to $\jinv$, we will get that $\t^{\delta_u}_v(t)$ is periodic for all $u\in\Fix(\iinv)\cup\Fix(\jinv)=\Vert(G)$.

We formulate this strategy as a proposition:
\begin{proposition}\label{prop:iinv_jinv}
 Suppose $G$ is an admissible $ADE$ bigraph, and assume that there are two symmetric unlabeled bigraphs 
 \[\B_\iinv:=(G,\iinv,V^\iinv_+,V^\iinv_0,V^\iinv_-,X_\iinv)\quad\text{and}\quad\B_\jinv:=(G,\jinv,V^\jinv_+,V^\jinv_0,V^\jinv_-,X_\jinv)\] 
 such that $V_0^\iinv\cup V^\jinv_0=\Vert(G)$. Let $\B^*_\iinv=(G^*_\iinv,\dots)$ and $\B^*_\jinv=(G^*_\jinv,\dots)$ be their duals. If $G^*_\iinv$ and $G^*_\jinv$ are tensor products of $ADE$ Dynkin diagrams then the tropical $T$-system for $G$ is periodic, that is, for all $u,v\in\Vert(G)$ one has
 \[\t^{\delta_u}_v(t+2N)=\t^{\delta_u}_v(t),\]
 and, moreover, $N=h(G)+h'(G)$. 
\end{proposition}
\begin{proof}
 Follows from Proposition~\ref{prop:coxeter_symmetric} and Corollary~\ref{cor:duality} using an observation above that if $u\in\Fix(\iinv)$ (resp., $u\in\Fix(\jinv)$) then $\delta_u$ is an $\iinv$-symmetric (resp., $\jinv$-symmetric) labeling.
\end{proof}

\subsubsection{The case of $(A^{m-1}D)_n$}

Let $\B_\iinv=((A^{m-1}D)_n,\i,V^\iinv_+,V^\iinv_0,V^\iinv_-,X_\iinv)$ be as follows. Distinguish two common vertices $v_1^+$ and $v_1^-$ of the two single color subdiagrams of type $D$, see Figure \ref{fig:zper17} for an example.  
   \begin{figure}
    \begin{center}
\vspace{-.1in}
\scalebox{0.52}{
\input{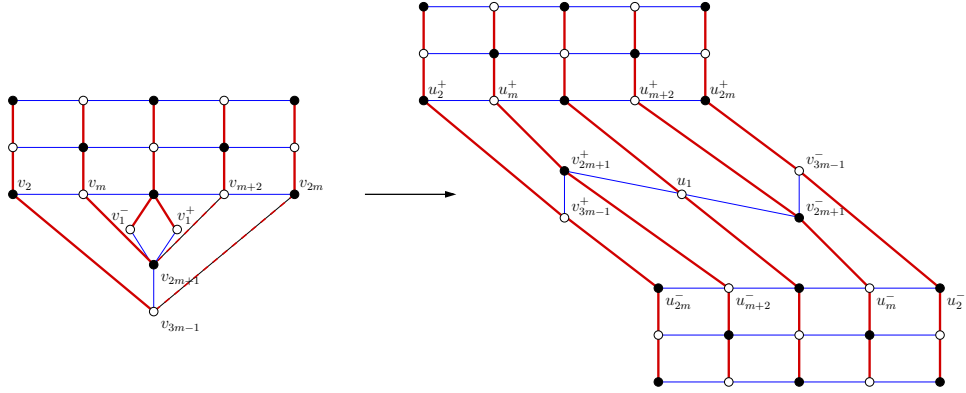} 
}
\vspace{-.1in}
    \end{center} 
    \caption{The reduction of $(A^{m-1}D)_n$ to $A_{2n-1} \otimes A_{2m-1}$. The edges in $X_\iinv$ are dashed.}
    \label{fig:zper17}
\end{figure}
Define $\iinv$ to be the involution that swaps $v_1^+$ and $v_1^-$ and fixes all other vertices. Denote the rest of the vertices of blue $D_{m+1}$ by $v_{2m+1}$ through $v_{3m-1}$, and  denote the adjacent vertices of the blue copy of $A_{2m-1}$ by $v_2$ through $v_{2m}$.
Finally, let 
\[X_\iinv = \{(v_{m+2}, v_{2m+1}), \ldots, (v_{2m}, v_{3m-1})\}.\] 
The construction of the dual $b^*$ is shown in Figure \ref{fig:zper17}. It is clear that one gets $A_{2n-1} \otimes A_{2m-1}$. 

In order to apply Proposition~\ref{prop:iinv_jinv}, we need to define an involution $\jinv$. It flips the graph in Figure~\ref{fig:zper17} on the left horizontally except that it fixes $v_1^+$ and $v_1^-$. In this case, $X_\jinv$ consists of only one edge $(v_{2m+1},v_1^+)$ and the dual symmetric unlabeled bigraph will be of type $D_{n+1}\otimes D_{m+1}$, so we are done.

Alternatively, we could avoid introducing $\jinv$ as follows. By Corollary \ref{cor:duality}, we know that 
\[\t^{\delta_u}_v(\e_v+2(h+h'))=\t^{\delta_u}_v(\e_v)\] 
holds for all fixed points $u$ of $\iinv$.

It remains to show the periodicity with initial conditions $\delta_{v_1^+}$ and $\delta_{v_1^-}$. We can use the above reduction to $A_{2m-1} \otimes A_{2n-1}$ again if we change either of those initial conditions into $(\delta_{v_1^+} + \delta_{v_1^-})/2$. It is easy to see 
this has no effect on the tropical system outside of vertices $v_1^+$ and $v_1^-$, from which one easily observes the desired behavior.

\subsubsection{The case of $(AD^{m-1})_n$}

Let $\B_\iinv=((AD^{m-1})_n,\iinv,V^\iinv_+,V^\iinv_0,V^\iinv_-,X_\iinv)$ be as follows.
   \begin{figure}
    \begin{center}
\vspace{-.1in}
\scalebox{0.8}{
\input{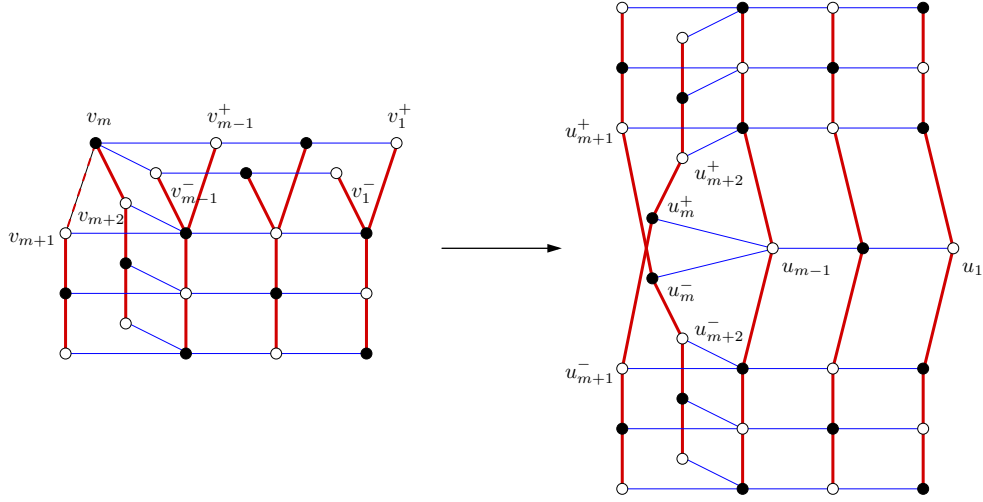} 
}
\vspace{-.1in}
    \end{center} 
    \caption{The reduction of $(AD^{m-1})_n$ to $A_{2n-1} \otimes D_{m+1}$. The edges in $X_\iinv$ are dashed.}
    \label{fig:zper18}
\end{figure}
Define $\iinv$ (resp., $\jinv$) to be the involution that swaps the exceptional vertices of all red (resp., blue) subgraphs $D_{n+1}$ (resp., $D_{m+1}$) and fixes the rest of the vertices (compare with Figure \ref{fig:zper13}). Because of the symmetry, it is enough to consider only the case of $\iinv$. 
Let $v_1^{+}, v_1^-, \ldots, v_{m-1}^+, v_{m-1}^-$ be the vertices swapped by $\iinv$, see Figure \ref{fig:zper18}. Let $v_m$ be the remaining vertex on this copy of $A_{2m-1}$, and let $v_{m+1}$ and $v_{m+2}$ be the two vertices
adjacent to $v_m$ by red edges. Let 
\[X_\iinv = \{(v_m, v_{m+1})\}.\] 
The construction of the dual $\B_\iinv^*$ is shown in Figure \ref{fig:zper18}. It is clear that one gets $A_{2n-1} \otimes D_{m+1}$. Thus we are done by Proposition~\ref{prop:iinv_jinv}

\subsubsection{The case of $EE^{n-1}$}

Let $\B_\iinv=(EE^{n-1},\iinv,V^\iinv_+,V^\iinv_0,V^\iinv_-,X_\iinv)$ be as follows.
   \begin{figure}
    \begin{center}
\vspace{-.1in}
\scalebox{0.7}{
\input{zper20.pstex_t} 
}
\vspace{-.1in}
    \end{center} 
    \caption{The reduction of $EE^{n-1}$ to $A_{2n-1} \otimes E_6$. The edges in $X_\iinv$ are dashed.}
    \label{fig:zper20}
\end{figure}
We can distinguish one special copy of $E_6$ and $n-1$ non-special copies of $E_6$ in the obvious way. 
We are going to make use of the involution for $E_6$ in Figure \ref{fig:zper14}. The involution $\iinv$ acts on the special copy of $E_6$. The involution $\jinv$ acts on the $n-1$ non-special copies of $E_6$. Let $v_1, v_2, v_3^+, v_3^-,v_4^+, v_4^-$ be the vertices of the special copy 
of $E_6$, see Figure \ref{fig:zper20}. Let $v_5$ and $v_6$ be the vertices adjacent to $v_1$ and $v_2$, as in Figure \ref{fig:zper20}. Finally, choose 
\[X_\iinv = \{(v_1, v_5),(v_2, v_6)\}.\]
It is clear that the dual $\B_\iinv^*$ is of type $A_{2n-1} \otimes E_6$.

Now, let $\B_\jinv=(EE^{n-1},\jinv,V^\jinv_+,V^\jinv_0,V^\jinv_-,X_\jinv)$ be as follows.
   \begin{figure}
    \begin{center}
\vspace{-.1in}
\scalebox{0.65}{
\input{zper19.pstex_t} 
}
\vspace{-.1in}
    \end{center} 
    \caption{The reduction of $EE^{n-1}$ to $D_{n+1} \otimes E_6$. The edges in $X_\jinv$ are dashed.}
    \label{fig:zper19}
\end{figure}
Let $v_1^+$, $v_1^-$, $\ldots$, $v_{2n-2}^+$, $v_{2n-2}^-$ be the vertices not fixed by $\jinv$, see Figure \ref{fig:zper19}. Let $v_{2n-1}$ and $v_{2n}$ be the remaining two vertices in the corresponding copies of $A_{2n-1}$, and let $v_{2n+1}$ be one of the remaining vertices connected to $v_{2n}$.
Choose 
\[X_\jinv = \{(v_{2n}, v_{2n+1})\}.\] 
It is clear that the dual $\B_\jinv^*$ is $D_{n+1} \otimes E_6$, so the proof now again follows from Proposition~\ref{prop:iinv_jinv}.

\appendix

\section{Stembridge's classification, exceptional cases} \label{sec:app}

Figures \ref{fig:zper21}, \ref{fig:zper22}, \ref{fig:zper23}, \ref{fig:zper24} show the eight exceptional cases of Stembridge's classification, preserving his notation. The figures for the cases $D_6*D_6,\ E_8*E_8,\ D_5 \boxtimes A_7,\ E_7 \boxtimes D_{10},$ and $E_8 * E_8 \equiv E_8 \equiv E_8$ are borrowed from~\cite{S}.
   \begin{figure}
    \begin{center}
\vspace{-.1in}
\scalebox{0.7}{
\input{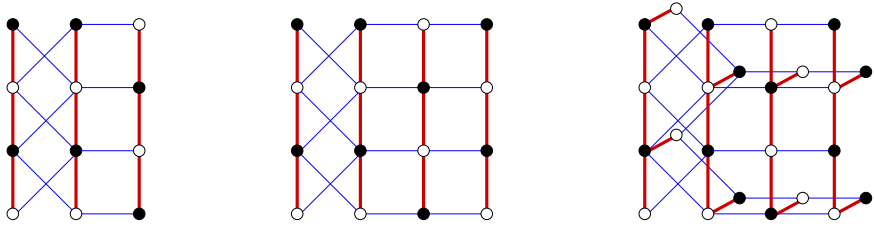} 
}
\vspace{-.1in}
    \end{center} 
    \caption{The cases $D_6 * D_6$, $E_8 * E_8$, and $E_8 * E_8 \equiv E_8$.}
    \label{fig:zper21}
\end{figure}
   \begin{figure}
    \begin{center}
\vspace{-.1in}
\scalebox{0.9}{
\input{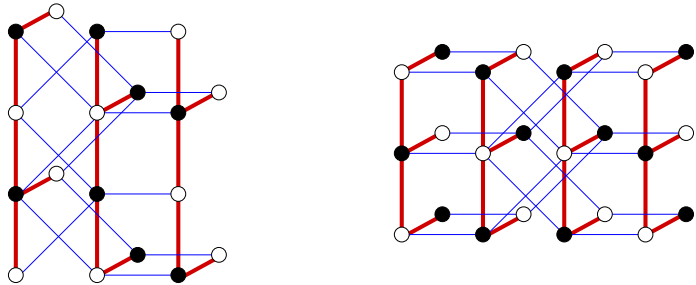} 
}
\vspace{-.1in}
    \end{center} 
    \caption{The cases $D_6 * D_6 \equiv D_6$  and $E_6 \equiv E_6 * E_6 \equiv E_6$.}
    \label{fig:zper22}
\end{figure}
   \begin{figure}
    \begin{center}
\vspace{-.1in}
\scalebox{0.8}{
\input{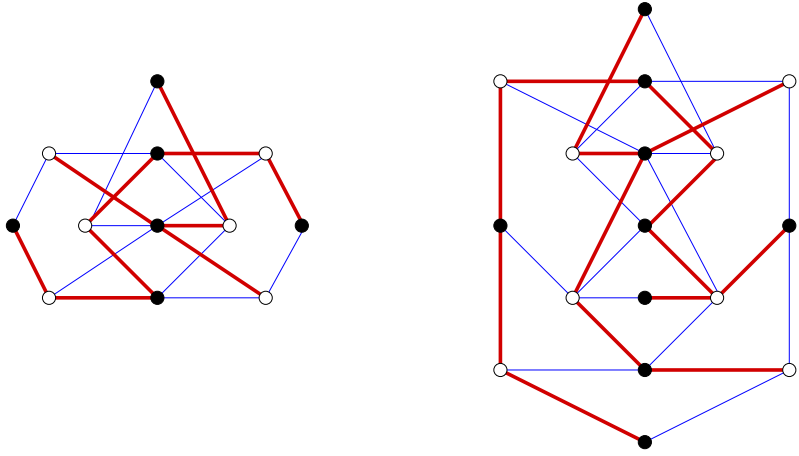} 
}
\vspace{-.1in}
    \end{center} 
    \caption{The cases $D_5 \boxtimes A_7$  and $E_7 \boxtimes D_{10}$.}
    \label{fig:zper23}
\end{figure}
   \begin{figure}
    \begin{center}
\vspace{-.1in}
\scalebox{0.8}{
\input{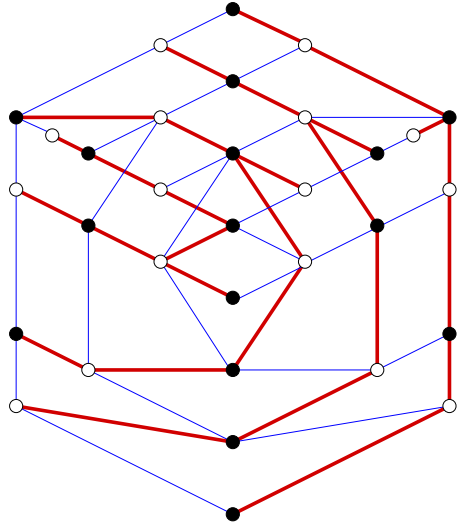} 
}
\vspace{-.1in}
    \end{center} 
    \caption{The case $E_8 * E_8 \equiv E_8 \equiv E_8$.}
    \label{fig:zper24}
\end{figure}

\begin{remark}\label{remark:computation}
 To verify validity of our theorem in each of the eight cases, it suffices to show that \[\t^{\delta_u}_v(\e_v+2(h+h'))=\t^{\delta_u}_v(\e_v)\] for arbitrary vertices $u,v\in\VertQ$. This can easily be done and is not computation intensive by computer standards. Indeed, for each exceptional quiver $Q$ with $d$ vertices, one needs to perform $O(d^2\cdot 2(h(Q)+h'(Q)))$ arithmetic operations. For the largest graph $E_8 * E_8 \equiv E_8 \equiv E_8$, we have $d=32$ and $h=h'=30$. Checking periodicity for all the eight cases takes less than a second on a standard office computer. 
 
 For example, one verification for the case of $D_6 * D_6$ would run as in Figure \ref{fig:zper25}. A complete verification for the case $A_3\otimes A_1$ was done in Example~\ref{example:A3_A1}.
\end{remark}

   \begin{figure}
    \begin{center}
\vspace{-.1in}
\scalebox{0.65}{
\input{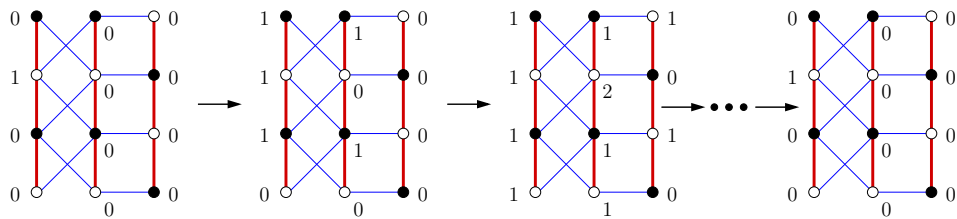} 
}
\vspace{-.1in}
    \end{center} 
    \caption{Verifying the periodicity of $\t^{\delta_u}$ for one particular choice of $u$ for $D_6 * D_6$ in $30=2(10+5)$ steps.}
    \label{fig:zper25}
\end{figure}

\bibliographystyle{plain}
\bibliography{zperiod}

\end{document}